\documentclass[11pt,letterpaper]{article}

\usepackage[T1]{fontenc}
\usepackage{lmodern}
\usepackage{microtype}
\usepackage{amsmath,amssymb,amsthm,mathtools}
\usepackage[margin=1.03in]{geometry}
\usepackage{graphicx}
\usepackage{xcolor}
\usepackage{float}
\usepackage{placeins}
\usepackage{tikz}
\usetikzlibrary{arrows.meta,calc,positioning}
\definecolor{DarkTurquoise}{HTML}{006D77}
\definecolor{ReferenceViolet}{HTML}{6F2DA8}
\usepackage[
  backend=biber,
  style=alphabetic,
  sorting=nyt,
  giveninits=true,
  maxbibnames=99,
  doi=false,
  isbn=false,
  url=false,
  eprint=false
]{biblatex}
\AtBeginBibliography{\setlength{\emergencystretch}{2em}}
\usepackage{xurl}
\usepackage{hyperref}
\hypersetup{
  colorlinks=true,
  linkcolor=DarkTurquoise,
  citecolor=ReferenceViolet,
  urlcolor=DarkTurquoise,
  pdfauthor={Pablo Portilla Cuadrado},
  pdftitle={A reduction theorem for Lê's conjecture}
}
\usepackage{aliascnt}
\usepackage[capitalise,nameinlink,noabbrev]{cleveref}

\numberwithin{equation}{section}
\makeatletter
\let\c@figure\@undefined
\makeatother
\newaliascnt{figure}{equation}
\newaliascnt{theorem}{equation}
\newtheorem{theorem}[theorem]{Theorem}
\aliascntresetthe{theorem}
\newaliascnt{proposition}{equation}
\newtheorem{proposition}[proposition]{Proposition}
\aliascntresetthe{proposition}
\newaliascnt{lemma}{equation}
\newtheorem{lemma}[lemma]{Lemma}
\aliascntresetthe{lemma}
\newaliascnt{corollary}{equation}
\newtheorem{corollary}[corollary]{Corollary}
\aliascntresetthe{corollary}
\theoremstyle{definition}
\newaliascnt{definition}{equation}
\newtheorem{definition}[definition]{Definition}
\aliascntresetthe{definition}
\newaliascnt{notation}{equation}
\newtheorem{notation}[notation]{Notation}
\aliascntresetthe{notation}
\crefname{notation}{notation}{notations}
\Crefname{notation}{Notation}{Notations}
\newaliascnt{remark}{equation}
\theoremstyle{remark}
\newtheorem{remark}[remark]{Remark}
\aliascntresetthe{remark}
\theoremstyle{plain}

\newcommand{\C}{\mathbb{C}}
\newcommand{\ord}{\operatorname{ord}}
\newcommand{\mult}{\operatorname{mult}}

\title{A reduction theorem for L\^e's conjecture}
\author{Pablo Portilla Cuadrado}
\date{24 August 2026}

\AtBeginDocument{
  \setlength{\abovedisplayskip}{9pt plus 3pt minus 6pt}
  \setlength{\belowdisplayskip}{9pt plus 3pt minus 6pt}
}

\begin{document}
\maketitle

\begin{abstract}
 Let
$\boldsymbol{n}=(n_1,n_2,n_3):(\C^2,0)\to(\C^3,0)$ be a holomorphic map
germ admitting an injective representative.  We prove that
if $(d\boldsymbol{n})_0=0$, then
$\ord_0(\boldsymbol{n})\in\{2,3,4\}$.  This reduces L\^e's conjecture to
excluding potential counterexamples $\boldsymbol{n}$ of orders $2$, $3$, and $4$. The proof
combines techniques from the theory of plane curve singularities, explicit
cobordism constructions, and genus bounds obtained by applying properties of the $\Upsilon$ invariant
coming from knot Floer homology.
\end{abstract}

\section{Introduction}

\subsection*{L\^e's conjecture and the Reduction Theorem}

Let $(X,0)\subset(\C^N,0)$ be a reduced surface germ, where $N\geq2$. For
$0<\epsilon$ sufficiently small, let
$
 L_X=X\cap \mathbb{S}^{2N-1}_\epsilon
$
be its link.
Mumford \cite{Mumford} proved that a normal surface germ is smooth if and only if
its link is simply connected, or equivalently, homeomorphic to $\mathbb{S}^3$.
Nonnormal surface germs behave differently.
For example,   $\{z^2-x^3=0\}\subset\C^3$ is a non smooth constant family
of cusps and has link homeomorphic to $\mathbb{S}^3$.

L\^e conjectured that if the link of a reduced surface germ
$(X,0)\subset(\C^3,0)$ is homeomorphic to $\mathbb{S}^3$, then this is always the case, that is, $(X,0)$ is
the total space of an equisingular deformation of an irreducible plane
curve singularity. Here equisingularity means that the embedded topological
type of the fibers is constant. This problem was recorded in the
Proceedings of the 1982 Plans-sur-Bex seminar
\cite[p.~254]{Weber}.

There are several equivalent reformulations of the problem. One can show that the link is homeomorphic to
$\mathbb{S}^3$ if and only if the normalization is a bijection from a smooth
surface germ. After identifying this germ with $(\C^2,0)$, the
conjecture is equivalent to saying that every holomorphic map
germ
\[
 \boldsymbol n:(\C^2,0)\longrightarrow(\C^3,0)
\]
which admits an injective representative, satisfies
$
 \operatorname{rank}(d\boldsymbol n)_0\geq1.
$
These reformulations are discussed in \cite[Fact~1 and the discussion on p.~346]{Bobadilla}.

Let$(n_1,n_2,n_3)$ be the coordinate functions of $\boldsymbol n$. The order of $\boldsymbol n$ is defined by
$
\ord_0(\boldsymbol n)
 =\min_{1\leq i\leq3}\ord_0(n_i).
$
Our main result uses this formulation and excludes every possible counterexample to the conjecture with  $ \ord_0(\boldsymbol n) \geq 5$. More concretely:

\begin{theorem}
\label{thm:reduction}
Let
$\boldsymbol n:(\C^2,0)\to(\C^3,0)$ be a holomorphic map germ
that admits an injective representative. If
$(d\boldsymbol n)_0=0$, then
$
\ord_0(\boldsymbol n)\in\{2,3,4\}.
$
\end{theorem}

We recall that $(d\boldsymbol n)_0=0$ if and only if
$\ord_0(\boldsymbol n)\geq2$. So the conjecture is reduced to excluding
counterexamples of orders $2$, $3$, and $4$. 
\subsection*{Earlier approaches}

Our setting is similar to that of N\'emethi
\cite{Nemethi}. Let
$f:(\C^2,0)\to(\C^3,0)$ be an injective germ with reduced image germ
$(X,0)$. Let $w_1,w_2,w_3$ be coordinates in $(\C^3,0)$, with $f_i=w_i\circ f$,
so that
\[
 V_0=f_1^{-1}(0) \quad \text{ and } \quad W_0=X\cap\{w_1=0\}
\]
are reduced plane curve germs. N\'emethi compared these germs and
studied the Milnor fibers $V_t=f_1^{-1}(t)$ and their
(singular) images by the map $(f_2,f_3)$. Under an additional analytic hypothesis on $X$,
he used versal deformations to argue that the image of the vanishing cycles of $V_t$
and the vanishing cycles at the singularities of $W_t$ form
two orthogonal parts of a geometric basis of vanishing cycles for $W_0$ which would contradict the fact that Dynking diagrams associated with distinguished bases are connected.
Keilen and Mond \cite{KeilenMond} found a gap in this argument: for the injective map
$f(x,y)=(y^3+x^2,x,y^2)$, the parameter curve in the base of the versal deformation lies in the discriminant and cannot be perturbed as required by N\'emethi's construction.

Luengo and Pichon \cite[Theorem~5.1]{LuengoPichon} proved the conjecture for cyclic covers of normal
surface germs totally branched over the zero locus of a holomorphic
function. Their proof uses an explicit description of the
minimal Waldhausen decomposition of the link. 
Fern\'andez de Bobadilla \cite[Theorem~3 and Conjecture~6]{Bobadilla}
proved the conjecture in the case when the image surface contains a smooth curve passing
through the origin. He also reformulated the problem in terms of
plane curve singularities and their relative polar curves. Much of the setting developed in \cite{Bobadilla} is recalled and used here, although we don't use his plane curve reformulation.
Pe\~nafort Sanchis \cite[Proposition~4.3]{PenafortSanchis}
(Corollary~4.4 in arXiv:1609.03222v3) proved the conjecture for
reflection maps by establishing a corank bound.
Reflection maps are compositions of embeddings with orbit
maps of finite complex reflection groups. In  \cite{Hepler}, Hepler used the language of perverse
sheaves to express bijectivity of the normalization in terms of the
vanishing of the comparison complex. And phrased L\^e's conjecture
via the vanishing cycles of a generic linear form. He also
recovered the case in which the image contains a smooth curve.

\subsection*{Strategy and organization of the proof}

The proof has three parts, corresponding to
\Cref{sec:plane-sections,sec:knot-cobordisms,sec:numerical-estimates}. Assume $(d\boldsymbol n)_0=0$. Let $X$ be a small representative of the reduced image germ of
$\boldsymbol n$, and let $\Sigma=\operatorname{Sing}(X)$ be its singular locus which is a complex curve.
In \Cref{sec:plane-sections}, we choose a generic linear form
$\lambda_1:\C^3\to\C$ satisfying certain properties. On the Milnor representatives fixed in
\Cref{subsec:general-plane-sections}, we put
\[
 v=\lambda_1\circ\boldsymbol n,\qquad
 V_s=\{v=s\},\qquad
 W_s=X\cap\{\lambda_1=s\}.
\]
Write $V=V_0$ and $W=W_0$ for this discussion. For $s\neq0$ small enough,
let
\[
 \operatorname{Sing}(W_s)=\{Q_1(s),\ldots,Q_N(s)\}
\]
be the singular points of the nearby fibers of $\lambda_1|_{X}$ which satisfy $\operatorname{Sing}(W_s)=\Sigma \cap W_s$. We set $\boldsymbol n_s=\boldsymbol n|_{V_s}$.
Then \Cref{prop:plane-sections} identifies
$\boldsymbol n_s:V_s\to W_s$ as the normalization map. The singular
germs $(W_s,Q_j(s))$ are the transverse plane branches of $X$
described before \Cref{def:transverse-type}.
The Milnor number identity in
\cite[equation~(4)]{Bobadilla} yields the analogous equation in terms of the delta invariants (see \eqref{eq:delta-specialization}):
\[
 \delta(W,0)=\delta(V,0)+
 \sum_{j=1}^N\delta\bigl(W_s,Q_j(s)\bigr).
\]
Let $b$ be the number of branches of $V$ (and also of $W$), and set
\[
 m_{\boldsymbol n}=\ord_0(\boldsymbol n),\qquad
 m_X=\mult_0(X),\qquad
 M_{\boldsymbol n}=m_X-2\delta(V,0).
\]
Denote by
$\operatorname{Sing}(X)=\bigcup_{i=1}^c\Sigma_i$ the decomposition of the singular locus, and let
$\Gamma_i=(\boldsymbol n^{-1}(\Sigma_i))_{\mathrm{red}}$. Let
$m_i$ be the multiplicity of the transverse type along $\Sigma_i$.
We define $m_{\mathrm{tr}}=\max_i m_i\geq2$. This section gives the inequalities
\[
 \sum_{i=1}^c(m_i-1)i_0(v,\Gamma_i)
 \leq 2m_X-M_{\boldsymbol n}-b,\quad
 m_{\boldsymbol n}\leq M_{\boldsymbol n}, \quad \text{ and } \quad
 m_X\leq m_{\boldsymbol n}M_{\boldsymbol n}.
\]
Where $i_0(v,\Gamma_i)$ denotes the local intersection multiplicity at $0$ of $\{v=0\}$ anbd $\Gamma_i$.
These estimates are proven in
\Cref{prop:polar-estimate,prop:plane-estimates}.

In \Cref{sec:knot-cobordisms} we view the counterpart of the identity between delta numbers in the world of knot
cobordisms with controlled genus. Fix a small $s_0\neq0$, put
$Q_j=Q_j(s_0)$, and let $K_{Q_j}$ be the knot of the transverse
plane branch $(W_{s_0},Q_j)$. Let $K_{\mathrm{tr}}$ be the connected sum of the knots $(W_{s_0},Q_j)$, and let $\Lambda_W$ be the link of $W$.
In \Cref{prop:exterior}, we construct a link cobordism from $K_{\mathrm{tr}}$ to
$\Lambda_W$, using the boundary trivialization of the nearby sections and an embedded realization of $K_{\mathrm{tr}}$. In \Cref{prop:braid} we  construct a link cobordism (via a braid construction) from $\Lambda_W$ to the torus knot
$T(m_X,m_X+1)$. Gluing these surfaces, that is, composing the two cobordisms, we get another knot cobordism $F$
from $(\mathbb{S}^3,K_{\mathrm{tr}})$ to
$(\mathbb{S}^3,T(m_X,m_X+1))$
in a manifold diffeomorphic to $\mathbb{S}^3\times[0,1]$, with
\[
 g(F)=\delta(V,0)+\delta(W,0)-\frac{m_X(m_X-1)}2,
 \quad \text{ and } \quad
 g_3(K_{\mathrm{tr}})=\delta(W,0)-\delta(V,0).
\]
Where $g_3(K)$ denotes Seifert genus of the knot $K$ (see \Cref{cor:genus-identities}).

Finally, in \Cref{sec:numerical-estimates} we apply knot Floer invariants to this
cobordism. Let $\Upsilon$ be the invariant of Ozsv\'ath, Stipsicz, and Szab\'o
\cite[Section~3]{OSS}. For the function
\[
 \Phi_K(\tau)=\Upsilon_K(\tau)+g_3(K)\tau,
\]
the bound of Ozsv\'ath, Stipsicz, and Szab\'o in
\cite[Theorem~1.11]{OSS} gives
\[
 \Phi_{T(m_X,m_X+1)}(\tau)
 \leq\sum_{j=1}^N\Phi_{K_{Q_j}}(\tau)
      +(m_X-M_{\boldsymbol n})\tau,
 \qquad\text{for }\tau\in[0,1]\cap\mathbb Q.
\]
The shift in \eqref{eq:phi-definition} is inspired by Tru\"ol
\cite[Remark~4.14]{Truol}.
The comparison follows from a single application of
\Cref{lem:phi-cobordism}, as shown in \Cref{cor:upsilon-comparison}. Then, we evaluate this comparison at
$\tau_0=2/m_{\mathrm{tr}}$. Wang's theorem
\cite[Theorem~8]{Wang}
(Theorem~5 in arXiv:1505.06835v1) gives
$\Phi_{K_{Q_j}}(\tau_0)=0$ for every $1\leq j\leq N$.
This, together with an explicit formula \eqref{eq:torus-phi} for the $\Upsilon$ invariant for the torus knot gives, the inequalities 
\[
 M_{\boldsymbol n}\leq m_{\mathrm{tr}} \quad \text{ and } \quad
 m_X+M_{\boldsymbol n}\leq3m_{\mathrm{tr}}.
\]
These lower bounds on $m_{\mathrm{tr}}$ are proved in
\Cref{thm:estimates}. Combining the second with the upper bound from
\Cref{prop:polar-estimate} gives
\[
 \frac{m_X+M_{\boldsymbol n}}{3}
 \leq m_{\mathrm{tr}}
 \leq 1+\frac{2m_X-M_{\boldsymbol n}-b}{m_{\boldsymbol n}}.
\]
Comparing the two bounds and using $m_{\boldsymbol n}\leq M_{\boldsymbol n}$
and $b\geq1$ excludes $m_{\boldsymbol n}\geq6$.
The case $m_{\boldsymbol n}=5$ also requires the resolution inequality
$m_X\leq m_{\boldsymbol n}M_{\boldsymbol n}$.
This proves \Cref{thm:reduction}.

\subsection*{AI disclosure}

The principal mathematical ideas and proof strategies of this paper
were conceived and developed by the author. ChatGPT 5.6 assisted with bibliographic
searches, general discussions, and revising the writing and
presentation. The author takes responsibility for all the arguments in the final manuscript.

A first draft of this manuscript contained an argument for the case
$m_{\boldsymbol n}=4$, also by the author. When that draft was
handed to ChatGPT, it found a flaw in the argument.
After several iterations and a final long session,
ChatGPT found a fix for the case $m_{\boldsymbol n}=4$ and proved the conjecture, also, in the two remaining cases, $m_{\boldsymbol n}=2,3$. These cases will appear in another paper with a unified argument shortly after this one.
In that other paper the author has played
mainly the role of digesting, verifying, and clarifying the arguments,
and removing unnecessary references, reducing substantially the original manuscript produced by AI. However, the author has minor intellectual contribution to that work. This, together with the fact that the methods used there have a completely different flavor is the reason why that part of the work will appear in a different article.

\section{General plane sections and numerical estimates}
\label{sec:plane-sections}

Let $\boldsymbol{n}=(n_1,n_2,n_3):(\C^2,0)\to(\C^3,0)$ be a
holomorphic germ with an injective representative, and let
$(X,0)\subset(\C^3,0)$ be its reduced image germ which is well defined by injectivity. Choose a
representative $X\subset\C^3$ of $(X,0)$ and
\textbf{fix once and for all a compatible injective representative
$\boldsymbol n$.}
By abuse of notation we denote the representative $\boldsymbol n$ by the same symbol as the map germ. Set
\begin{equation}
 m_{\boldsymbol{n}}=\ord_0(\boldsymbol{n})
 =\min_{1\leq i\leq3}\ord_0(n_i).
 \label{eq:order-n}
\end{equation}
For a generic hyperplane $H \subset \C^3$ through the origin, one has
\begin{equation}
 \mult_0\bigl(\boldsymbol{n}^{-1}(H \cap X)\bigr)=m_{\boldsymbol{n}}.
 \label{eq:general-section-order}
\end{equation}
Indeed, if $H=\{\lambda=0\}$, then
$\boldsymbol{n}^{-1}(H \cap X)$ is defined by
$\lambda\circ\boldsymbol{n}$, and for general $\lambda$ the
 homogeneous part of this function of degree $m_{\boldsymbol{n}}$ is
not zero, so this curve germ has also multiplicity $m_{\boldsymbol{n}}$.

Let $\mathfrak m=(x,y)\subset\C\{x,y\}$ be the maximal ideal.
The map $\boldsymbol{n}:(\C^2,0)\to(X,0)$ is the finite bijective
normalization of the surface (see
\cite[Fact~1 and the discussion on p.~346]{Bobadilla}). By the definition of the differential (or the contangent map), we have
\begin{equation}
 (d\boldsymbol{n})_0=0\quad\Longleftrightarrow\quad
 n_i\in\mathfrak m^2\ \text{for every }i.
 \label{eq:differential}
\end{equation}
\textbf{We assume from now on that $\boldsymbol{(d\boldsymbol n)_0=0}$.} The singular locus
$
 \Sigma=\operatorname{Sing}(X)
$
is a curve germ at the origin.
Otherwise, since $X$ is a reduced surface hypersurface with at most an
isolated singularity, it would be normal by
\cite[Theorem~1.96(1) and Corollary~B.8.10]{GreuelLossenShustin}.
So its normalization $\boldsymbol n$ would then be an isomorphism from a smooth space,
contradicting $(d\boldsymbol n)_0=0$.
We endow $\Sigma$ with its reduced analytic structure.

\subsection{General plane sections}
\label{subsec:general-plane-sections}

In this section we choose a linear form
$\lambda_1:\C^3\to\C$ which satisfies some genericity conditions slightly stronger than those of
Fern\'andez de Bobadilla \cite[pp.~347--348]{Bobadilla}. We later choose a second linear form
$\lambda_2:\C^3\to\C$, linearly independent of $\lambda_1$,
satisfying another genericity condition. These choices are crucial to prove \Cref{prop:polar-estimate} and the inequalities
\eqref{eq:pointbounds}. 

\begin{notation}
\label{not:principalization}
The ideal
\[
I=(n_1,n_2,n_3)\subset\C\{x,y\}
\]
is $\mathfrak m$-primary because $\boldsymbol n^{-1}(0)=\{0\}$.
Let $\pi$ be the principalization of the ideal which is obtained
by point blowups (see for example Wall's construction
\cite[proof of Proposition~11.1.3(iv)]{Wall}).
So, at each step we blow up the points at which the weak transform
of $I$ is proper, and we stop when this transform is the unit ideal.
Let $\mathcal C$ be the whole set of centers.
For $p\in\mathcal C$, write
$I_p=(n_{1,p},n_{2,p},n_{3,p})$ for the weak transform of $I$ at
$p$, and set
\[
 m_{\boldsymbol n,p}
 =\ord_p(I_p)
 =\min_{1\leq i\leq3}\ord_p(n_{i,p}).
\]
By construction, the set $\mathcal C$, with the multiplicities
$m_{\boldsymbol n,p}$, is the {\em cluster of base points of $I$} (in the sense of 
\cite[pp.~253--254]{CasasAlvero}). For a nonzero linear form
$\lambda_1$, we set
\[
\begin{aligned}
	v&=\lambda_1\circ\boldsymbol n,
	& H_0&=\{\lambda_1=0\},\\
	(V,0)&=(\{v=0\},0),
	& (W,0)&=(X\cap H_0,0).
\end{aligned}
\]
Here $V$ and $W$ denote the central curve germs. For $p\in\mathcal C$, let $V_p$ denote the strict transform of
$V$ on the surface on which $p$ is a center.
\end{notation}
Next, we  show that we can choose a linear form $\lambda_1$ such that the following properties hold:
\begin{enumerate}
\setlength{\itemsep}{2pt}
\setlength{\parsep}{0pt}
\renewcommand{\theenumi}{CA}
\renewcommand{\labelenumi}{\makebox[\labelwidth][l]{(\theenumi)}}
\item\label{cond:casas-alvero}
The curve germ $(V,0)$ is reduced and
$\mult_p(V_p)=m_{\boldsymbol n,p}$ for every $p\in\mathcal C$.
The principalization $\pi$ is also an embedded resolution of $V$
(not necessarily a minimal one).
\setcounter{enumi}{0}
\renewcommand{\theenumi}{G\arabic{enumi}}
\item\label{cond:section-multiplicity}
 We have $\mult_0(W)=\mult_0(X)$.
\item\label{cond:isolated-critical-point}
The map germ $v:(\C^2,0)\to(\C,0)$ has an isolated critical
point.
\item\label{cond:singular-locus-section}
The form $\lambda_1$ does not vanish identically on any irreducible
component of $(\Sigma,0)$.
\item\label{cond:section-reduced}
The plane curve germ $(W,0)$ is reduced.
\end{enumerate}

\begin{remark}
\label{rem:generality-properties}
These conditions above are stated here in the form in which we use them in the text, but it is enough to impose
(\ref{cond:casas-alvero}) and
(\ref{cond:section-multiplicity}), since
(\ref{cond:casas-alvero}) implies
(\ref{cond:isolated-critical-point}),
(\ref{cond:singular-locus-section}) and (\ref{cond:section-reduced}) as we see now:
\begin{list}{}{%
\setlength{\leftmargin}{0pt}%
\setlength{\rightmargin}{0pt}%
\setlength{\itemindent}{0pt}%
\setlength{\listparindent}{0pt}%
\setlength{\labelwidth}{0pt}%
\setlength{\labelsep}{0pt}%
\setlength{\topsep}{3pt}%
\setlength{\itemsep}{3pt}%
\setlength{\parsep}{0pt}%
}
\item \textbullet\quad
$(\ref{cond:isolated-critical-point})$.
The assumption that
$(d\boldsymbol n)_0=0$, gives $dv_0=0$. So the
origin is a critical point of $v$ and it is isolated since $(V,0)$ is reduced.
\item \textbullet\quad
$(\ref{cond:singular-locus-section})$.
Suppose that $\lambda_1$ is $0$ on a whole component
$\Sigma_i$ of $\Sigma$. At a point
$p\in\Gamma_i=\boldsymbol n^{-1}(\Sigma_i)$ sufficiently close to
the origin, and with $p\ne0$,
the restriction $\boldsymbol n|_{\Gamma_i}$ is a local biholomorphism
onto $\Sigma_i$. In particular,
$(d\boldsymbol n)_p(T_p\Gamma_i)=T_{\boldsymbol n(p)}\Sigma_i$.
The differential has rank one. Indeed, it can't have rank two since it would make $X$ smooth
at $\boldsymbol n(p)$). Hence
$
 \operatorname{im}(d\boldsymbol n)_p
 =T_{\boldsymbol n(p)}\Sigma_i\subset\ker d\lambda_1
$
and $dv_p=d\lambda_1\circ(d\boldsymbol n)_p=0$,
contradicting
(\ref{cond:isolated-critical-point}).
\item \textbullet\quad
$(\ref{cond:section-reduced})$. Property~(\ref{cond:singular-locus-section})
gives $W\cap\Sigma=\{0\}$ for representatives small enough.
Since $\boldsymbol n$ is biholomorphic over $X\setminus\Sigma$,
it induces a biholomorphism
$V\setminus\{0\}\to W\setminus\{0\}$.
And so $W$ is reduced away from the origin since $V$ is reduced. A repeated irreducible
factor in a defining equation of $W$ would make it nonreduced along
an entire branch, so $W$ is reduced.
\end{list}
\end{remark}

Fern\'andez de Bobadilla \cite[pp.~347--348]{Bobadilla} also
chooses a generic linear form  $x$, playing the role
of our $\lambda_1$. His construction has properties
(\ref{cond:isolated-critical-point}),
(\ref{cond:singular-locus-section}) and (\ref{cond:section-reduced}).
The isolated singularity of $x\circ\boldsymbol n$
is stated explicitly, and the other properties are used in the
construction of the sections and in formula~(2) there.
We prove in the following lemma that we can additionally impose
(\ref{cond:casas-alvero}) and (\ref{cond:section-multiplicity})
and still remain in a nonempty Zariski open set.
We write $[\lambda]$ for the class of a nonzero linear form
$\lambda$ in $\mathbb P(\check{\C}^3)$, where
$\check{\C}^3$ denotes the dual of $\C^3$.

\begin{lemma}
\label{lem:general-linear-forms}
There is a nonempty Zariski open set
$\mathcal U\subset\mathbb P(\check{\C}^3)$, such that every linear
form $\lambda_1$ with $[\lambda_1]\in\mathcal U$ satisfies 
(\ref{cond:casas-alvero}) and
(\ref{cond:section-multiplicity})--(\ref{cond:section-reduced}).
\end{lemma}

\begin{proof}
By \Cref{rem:generality-properties}, we just need to find an  Zariski open set where
(\ref{cond:casas-alvero}) and (\ref{cond:section-multiplicity}) hold simultaneously.

We start with (\ref{cond:casas-alvero}). The existence of a (non empty) Zariski open subset on which
(\ref{cond:casas-alvero}) holds, follows from the proof of
\cite[Theorem~7.2.13]{CasasAlvero}, and from
\cite[Lemma~7.2.6]{CasasAlvero}.
We explain here how these arguments apply in our situation, that is, to the elements of
$\operatorname{span}_{\C}\{n_1,n_2,n_3\}$.
In the terminology of \cite[p.~127]{CasasAlvero}, the curve germ $V$
{\em goes sharply through the weighted cluster} $\mathcal C$ if
$\mult_p(V_p)=m_{\boldsymbol n,p}$ for every $p\in\mathcal C$ and
$V$ has no singular infinitely near points outside $\mathcal C$.
Here singular infinitely near points are understood in the sense of
\cite[\S3.8]{CasasAlvero}. The fact that there are no such points outside $\mathcal C$ means
that $\pi$ is an embedded resolution of $V$. Indeed, the strict
transform is smooth and meets the exceptional divisor transversely
away from its crossings. The elements of $\operatorname{span}_{\C}\{n_1,n_2,n_3\}$ that define curve germs going sharply through
$\mathcal C$ form a nonempty Zariski open subset.
Indeed, the ideals excluded in the cited proof do not contain
$I$, and so they do not contain its generating span.
These germs are reduced by \cite[Proposition~4.2.6(a)]{CasasAlvero},
so (\ref{cond:casas-alvero}) holds on a nonempty Zariski open subset
of $\mathbb P(\check{\C}^3)$.
For related treatment of this subject, we also refer the reader to
\cite[Theorem~11.1.4(iii) and its proof, and Corollary~11.1.5,
p.~323]{Wall} as an alternative reference.

Let $F_d$ be the lowest nonzero
homogeneous part of a reduced local equation of $X$, of degree
$d=\mult_0(X)$. So (\ref{cond:section-multiplicity}) holds
as long as  $H_0$ is not a plane component of the tangent cone
$\{F_d=0\}$, so that $F_d|_{H_0}\not\equiv0$.
There are only finitely many such plane components, so the condition
holds on the complement of a finite subset of
$\mathbb P(\check{\C}^3)$, which is nonempty and Zariski open.

These two Zariski open subsets
are dense, so their intersection is nonempty and Zariski open.
\end{proof}

Fix a linear form $\lambda_1$ with $[\lambda_1]\in\mathcal U$
as in \Cref{lem:general-linear-forms} from now on. Choose once and for
all a closed Euclidean Milnor ball $\mathbb{B}\subset\C^3$ for $(X,0)$,
centered at the origin and small enough that $\mathbb{B}\cap H_0$ is a
Milnor ball for $(W,0)$, that
$\Sigma\cap H_0\cap \mathbb{B}=\{0\}$, and that $v$ has no critical point in
$\boldsymbol n^{-1}(\mathbb{B})\setminus\{0\}$. Require also that
$\Sigma\cap \mathbb{B}$ and
$(\boldsymbol n^{-1}(\Sigma\cap \mathbb{B}))_{\mathrm{red}}$ are smooth away
from the origin. For $s\in\C$, set
\[
 \begin{gathered}
  H_s=\{\lambda_1=s\},\\
  V_s=\{v=s\}\cap\boldsymbol n^{-1}(\mathbb{B}),\qquad
  W_s=X\cap H_s\cap \mathbb{B},\\
  \boldsymbol n_s=\boldsymbol n|_{V_s}\colon V_s\longrightarrow W_s.
 \end{gathered}
\]
We  use $V=V_0$ and $W=W_0$ for these representatives
of the central germs. We also restrict $X$ and $\Sigma$ to their
intersections with $\mathbb{B}$, and $\boldsymbol n$ to
$\boldsymbol n^{-1}(\mathbb{B})$, keeping the same notation.

Since $\mathbb{B}\cap H_0$ is a Milnor ball for $(W,0)$ and
$\Sigma\cap H_0=\{0\}$, the hyperplane $H_0$ meets
$X_{\mathrm{reg}}\cap\partial \mathbb{B}$ transversely.
By openness of transversality and compactness, we can choose $\eta>0$ such that
\[
 X\cap\partial \mathbb{B}\cap\{|\lambda_1|\leq\eta\}
 \subset X_{\mathrm{reg}}
\]
and such that $H_s$ meets $X_{\mathrm{reg}}\cap\partial \mathbb{B}$ transversely
whenever $|s|\leq\eta$. See \Cref{fig:plane-sections}.
\begin{figure}[!ht]
\centering
\begin{minipage}{0.95\linewidth}
	\centering
	\input{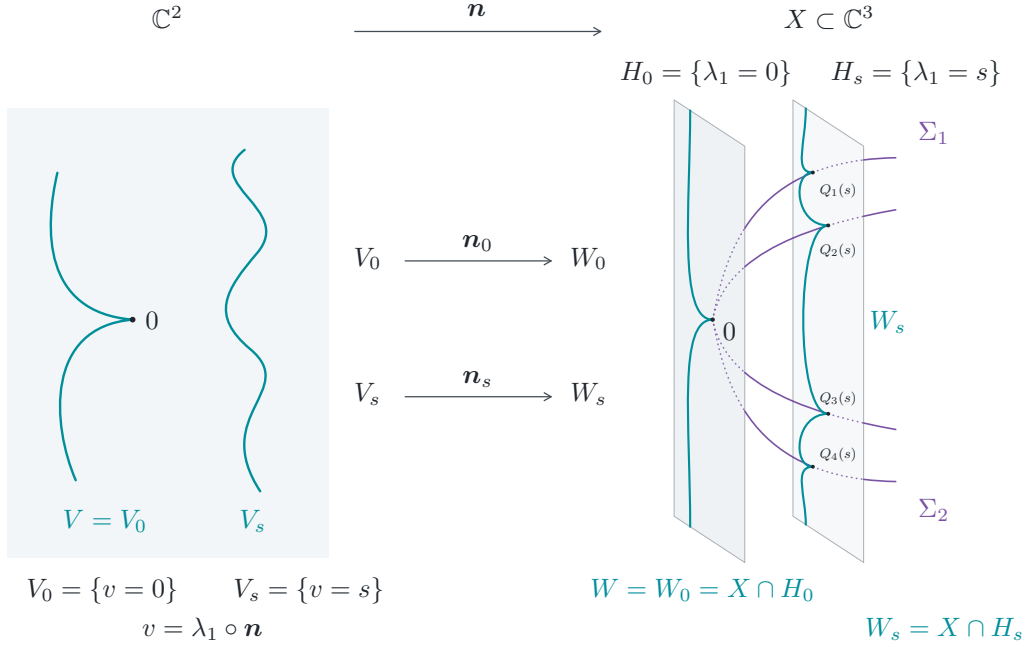}
\end{minipage}
\caption{Central and nearby fibers $V_s$ and $W_s$.}
\label{fig:plane-sections}
\end{figure}

At the origin we have
$m_{\boldsymbol n,0}=m_{\boldsymbol n}$, so
(\ref{cond:casas-alvero}) gives $\mult_0(V)=m_{\boldsymbol n}$.
If $(C,0)$ is a reduced plane curve germ, we denote by $\mu(C,0)$ and $\delta(C,0)$
its Milnor number and delta invariant respectively.

\begin{remark}
\label{rem:milnor-fibers}
For $s\ne0$ small enough, the surface $V_s$ is diffeomorphic to a Milnor
fiber of $v$, hence connected with $\chi(V_s)=1-\mu(V)$.
Indeed, choose a closed Euclidean Milnor ball
$\mathbb B_\epsilon\subset\operatorname{int}(\boldsymbol n^{-1}(\mathbb B))$
for $v$. The homeomorphism $\boldsymbol n_0:V\to W$ and the choices
of Milnor balls imply that
$V\setminus\operatorname{int}(\mathbb B_\epsilon)$ is a disjoint union
of annuli. For $s$ small enough, $V_s$ meets
$\partial\mathbb B_\epsilon$ transversely by the choice of Milnor ball.
At the outer boundary, transversality to $\boldsymbol n^{-1}(\partial\mathbb B)$
follows from the transversality of $H_s$ to
$X_{\mathrm{reg}}\cap\partial\mathbb B$, since $\boldsymbol n$ is
biholomorphic near these boundary intersections.
By compactness, property (\ref{cond:isolated-critical-point}) and
these transversality properties, Ehresmann's lemma for manifolds with boundary
identifies the corresponding parts of nearby fibers with these
annuli. So $V_s$ is obtained from
$\{v=s\}\cap\mathbb B_\epsilon$ by attaching annuli, or equivalently, boundary collars.
\end{remark}

Property~(\ref{cond:singular-locus-section}) gives
\begin{equation}
 \Sigma\cap\{\lambda_1=0\}=\{0\}.
 \label{eq:singular-section}
\end{equation}

\begin{proposition}
\label{prop:plane-sections}
The curve germs $(V,0)$ and $(W,0)$ have the same number
$b\geq1$ of branches. There is an integer $N\geq1$ such that, for every
nonzero $s$ small enough, the map
$
 \boldsymbol{n}_s:V_s\longrightarrow W_s
$
is the normalization, and
\begin{equation}
 \operatorname{Sing}(W_s)=\Sigma\cap H_s
 =\{Q_1(s),\ldots,Q_N(s)\}.
 \label{eq:singular-points}
\end{equation}
Moreover, each plane curve germ $(W_s,Q_j(s))$ is irreducible and
\begin{equation}
 \delta(W,0)=\delta(V,0)+
 \sum_{j=1}^{N}\delta\bigl(W_s,Q_j(s)\bigr).
 \label{eq:delta-specialization}
\end{equation}
\end{proposition}

\begin{proof}
The homeomorphism $\boldsymbol n|_V:V\to W$ induces in particular a bijection between their branches,
giving the first part of the statement.
By (\ref{cond:singular-locus-section}), the restriction $\lambda_1|_\Sigma$
is finite, and so over a punctured disk chosen small enough, it is an unramified
covering of some degree $N\geq1$.
This implies that  $\Sigma\cap H_s$ consists of exactly $N$ points for every
nonzero $s$ small enough.

For such $s$, condition~(\ref{cond:isolated-critical-point}) implies
that $V_s$ is smooth.
Since $\boldsymbol n$ is biholomorphic over $X\setminus\Sigma$,
we have that $W_s$ is smooth outside $\Sigma\cap H_s$, hence reduced.
The Jacobian criterion gives that the points in $\Sigma \cap H_s$ are actually singular so we get \eqref{eq:singular-points}.
The finite bijective map $\boldsymbol n_s$ is therefore a normalization map from the smooth curve $V_s$. In particular, the  preimage of a neighborhood of each singular point $Q_j(s)$ is a disk and
so each germ $(W_s,Q_j(s))$ is irreducible.
Fern\'andez de Bobadilla also deduces this irreducibility with a topological argument using that the link is a topological manifold together with the conical structure lemma \cite[proof of Fact~1, p.~347]{Bobadilla}.

Finally, \cite[equation~(4)]{Bobadilla} can be restated as
\[
 \mu(W,0)=\mu(V,0)+
 \sum_{j=1}^N\mu\bigl(W_s,Q_j(s)\bigr).
\]
Milnor's formula \cite[Theorem~10.5]{Milnor}, together with the already proven fact that $(W,0)$ and $(V,0)$ have the same
number $b$ of branches, and the irreducibility of $(W_s,Q_j(s)$,
yields \eqref{eq:delta-specialization}.
\end{proof}
\subsection{Transverse multiplicities and numerical estimates}

For plane curve germs $C,C'$ without common components, we denote by
$i_0(C,C')$ their intersection multiplicity at the origin.  We similarly write
\[
 i_0(h_1,h_2)=i_0(\{h_1=0\},\{h_2=0\}),\qquad
 i_0(h_1,C)=i_0(\{h_1=0\},C)
\]
for given holomorphic germs $h_1,h_2$ and a curve germ $C$.  The singular locus curve $\Sigma=\operatorname{Sing}(X)$ has a decomposition into irreducible branches that we denote by
\[
 \Sigma=\bigcup_{i=1}^c\Sigma_i
\]
and put
\[
 \Gamma=\bigl(\boldsymbol{n}^{-1}(\Sigma)\bigr)_{\mathrm{red}}
 =\bigcup_{i=1}^c\Gamma_i,
 \qquad
 \Gamma_i=\bigl(\boldsymbol{n}^{-1}(\Sigma_i)\bigr)_{\mathrm{red}}\subset\C^2.
\]
The curves $\Gamma_i$ are precisely the irreducible components of the
critical locus of $\boldsymbol{n}$. Indeed, $\boldsymbol n$ is
biholomorphic over $X\setminus\Sigma$ and the rank argument in
\Cref{rem:generality-properties} excludes rank two on points above $\Sigma$.
By our choice of representatives, the curves $\Sigma_i$ and $\Gamma_i$
are both smooth away from the origin. And, moreover, the map
$
 v|_{\Gamma_i}:
 \Gamma_i\setminus\{0\}\longrightarrow\C
$
is a submersion. Indeed, near  $p\in\Gamma_i\setminus\{0\}$,
the restriction $\boldsymbol n|_{\Gamma_i}$ is locally biholomorphic
and $(d\boldsymbol n)_p$ has rank one (since $p$ lies in the
critical locus). So
$(d\boldsymbol n)_p(T_p\Gamma_i)=\operatorname{im}(d\boldsymbol n)_p$.
The choice of $\mathbb{B}$ gives $dv_p\ne0$, so
$dv_p|_{T_p\Gamma_i}\ne0$.
Moreover, \eqref{eq:singular-section} and the bijectivity
of $\boldsymbol{n}$ give
$V\cap\Gamma_i=\{0\}$, so $i_0(v,\Gamma_i)$ is finite and not zero.
Now let $p\in\Gamma_i\setminus\{0\}$ be sufficiently close to the
origin and set $s=v(p)$.
Then $s\neq0$. By (\ref{cond:isolated-critical-point}),
the curve $V_s$ is smooth at $p$. The
submersion $v|_{\Gamma_i\setminus\{0\}}$ shows that $V_s$ meets
$\Gamma_i$ transversely at $p$.  Moreover, $H_s$ meets $\Sigma_i$ transversely
at $\boldsymbol{n}(p)$, and
\[
 (\boldsymbol{n}_s)_p:(V_s,p)\longrightarrow(W_s,\boldsymbol{n}(p))
\]
is the normalization of the transverse plane curve germ
of $X$ along $\Sigma_i$, by \Cref{prop:plane-sections}.
For each fixed $i$, the embedded topological type of the transverse plane
branch $(W_s,\boldsymbol{n}(p))$ is independent of
$p\in\Gamma_i\setminus\{0\}$, although it may depend on
$i$ (see
\cite[Part~I, \S\S1.4--1.5, pp.~40--41]{GoreskyMacPherson} and see also \cite[proof of Fact~1, p.~347]{Bobadilla}).

\begin{definition}
\label{def:transverse-type}
We call this embedded topological type the transverse
type along $\Sigma_i$. We denote the multiplicity of the transverse
type along $\Sigma_i$ by $m_i$.
\end{definition}

The map $\boldsymbol{n}$ maps
$\Gamma_i$ biholomorphically onto $\Sigma_i$ outside the origin.  Therefore
$(d\boldsymbol{n})_p$ is nonzero on $T_p\Gamma_i$.  The point $p$
is in the critical locus of $\boldsymbol{n}$, so
$(d\boldsymbol{n})_p$ has rank one and
\[
 \operatorname{im}(d\boldsymbol{n})_p=T_{\boldsymbol{n}(p)}\Sigma_i.
\]
Since $H_s$ is transverse to $\Sigma_i$, we have that
$T_{\boldsymbol{n}(p)}H_s\cap T_{\boldsymbol{n}(p)}\Sigma_i=\{0\}$.
The image of $(d\boldsymbol{n}_s)_p$ lies in this intersection and so it
vanishes. So the transverse plane branch is singular and $m_i\geq2$.
For every $s\neq0$ small enough, the fiber $V_s$
meets $\Gamma_i$ transversely for all $i$. By
\eqref{eq:singular-points}, the images of these intersection points are
\[
 \operatorname{Sing}(W_s)
 =\{Q_1(s),\ldots,Q_N(s)\}.
\]
If $p\in V_s\cap\Gamma_i$, then
$\mult_{\boldsymbol{n}(p)}(W_s)=m_i$.  We define
\begin{equation}
 m_{\mathrm{tr}}=\max_{1\leq i\leq c}m_i
 =\max_{1\leq j\leq N}\mult_{Q_j(s)}(W_s)\geq2.
 \label{eq:transverse-multiplicities}
\end{equation}
By
(\ref{cond:section-multiplicity}),
$\mult_0(W)=\mult_0(X)$.  We define the quantities
\begin{equation}
 m_X=\mult_0(X)=\mult_0(W),\qquad
 M_{\boldsymbol{n}}=m_X-2\delta(V,0).
 \label{eq:section-invariants}
\end{equation}
The motivation to name these quantities comes from
\eqref{eq:resolution-multiplicity} and \eqref{eq:resolution-sums}.
\begin{proposition}
\label{prop:polar-estimate}
We have the inequalities
\begin{gather*}
 \sum_{i=1}^c (m_i-1)i_0(v,\Gamma_i)
 \leq 2m_X-M_{\boldsymbol{n}}-b,\\
 m_{\mathrm{tr}}
 \leq 1+\frac{2m_X-M_{\boldsymbol n}-b}{m_{\boldsymbol n}}.
\end{gather*}
\end{proposition}

\begin{proof}
Choose a linear form $\lambda_2:\C^3\to\C$, linearly independent of
$\lambda_1$ and such that the line
\[
 L=\ker(\lambda_1)\cap\ker(\lambda_2)
\]
is not tangent to $W$. This choice is possible since $W$ has
only finitely many tangent lines at the origin.
Put $u=\lambda_2\circ\boldsymbol n$ and $U=\{u=0\}$.
Denote by $V_1,\ldots,V_b$ the branches of $V$ at the origin, and by
$\gamma_j:(\C,0)\to(V_j,0)$ some parametrizations (normalizations).
Since $\boldsymbol n|_V$ is finite and injective, the corresponding
branches of $W$ are $W_j=\boldsymbol n(V_j)$, with parametrizations
(normalizations) $\eta_j=\boldsymbol n\circ\gamma_j$.
By (\ref{cond:casas-alvero}) and (\ref{cond:section-reduced}), the curves
$V$ and $W$ are reduced. So
\[
 i_0(v,u)
 =\sum_{j=1}^b\ord_0(u\circ\gamma_j)
 =\sum_{j=1}^b\ord_0(\lambda_2\circ\eta_j)
 =i_0(W,L)=m_X,
\]
where the last equality uses that $L$ is not a tangent (by construction) and
(\ref{cond:section-multiplicity}). In particular, the map germ $(v,u)$
is finite. For $\delta>0$ small enough and $0<|s|\ll\delta$,
we can choose (using (\ref{cond:isolated-critical-point})),
a  Milnor representative $V'_s\subset V_s$ with
$u:V'_s\to\{z\in\C:|z|\leq\delta\}$ a proper branched covering which is
unramified on the boundary. Its degree is $i_0(v,u)=m_X$ by
conservation of local intersection number, and $\chi(V'_s)=1-\mu(V)$.
If we denote by $e_p$ the local degree of $u|_{V'_s}$ at $p$,
Riemann--Hurwitz gives
\[
 \sum_{p\in V'_s}(e_p-1)
 =m_X-\chi(V'_s)=m_X+\mu(V)-1.
\]
By (\ref{cond:singular-locus-section}), $V\cap\Gamma_i=\{0\}$
for each $i$. The conservation of local intersection number and the
transversality of $V_s$ and $\Gamma_i$ for $s\ne0$ give
\[
 \#(V'_s\cap\Gamma_i)=i_0(v,\Gamma_i).
\]
At each such point $p\in V'_s\cap\Gamma_i$, the restriction of
$\boldsymbol n$ to $V'_s$ parametrizes a transverse plane branch of
multiplicity $m_i$ at $Q=\boldsymbol n(p)$ (recall
\Cref{prop:plane-sections}). Since $\lambda_2$ is linear and its
restriction to this branch is nonconstant,
\[
 e_p=\ord_p\bigl(u|_{V'_s}-u(p)\bigr)\geq m_i,
\]
where equality holds if and only if the affine line $Q+L$ is not tangent to $W_s$
at $Q$. Grouping these lower bounds by branch $\Gamma_i$ we get a term $\sum_{i=1}^c(m_i-1)i_0(v,\Gamma_i)$. And discarding
any possible ramification over the smooth locus of $W_s$, we get the inequality
\[
 \sum_{i=1}^c(m_i-1)i_0(v,\Gamma_i)
 \leq m_X+\mu(V)-1
 =2m_X-M_{\boldsymbol n}-b.
\]
Where the last equality follows from Milnor's formula
$\mu(V)=2\delta(V,0)-b+1$ and the definition of
$M_{\boldsymbol n}$ in \eqref{eq:section-invariants}.

To obtain the upper bound on $m_{\mathrm{tr}}$, choose $i$ with
$m_i=m_{\mathrm{tr}}$. As observed above, $V\cap\Gamma_i=\{0\}$,
so the local intersection inequality and
\eqref{eq:general-section-order} give
\[
 i_0(v,\Gamma_i)
 \geq\mult_0(V)\mult_0(\Gamma_i)\geq m_{\boldsymbol n}.
\]
Since $m_{\mathrm{tr}}\geq2$, retaining the term corresponding to $i$
in the first inequality of the proposition yields
\[
 m_{\boldsymbol n}(m_{\mathrm{tr}}-1)
 \leq(m_{\mathrm{tr}}-1)i_0(v,\Gamma_i)
 \leq2m_X-M_{\boldsymbol n}-b.
\]
Dividing by $m_{\boldsymbol n}>0$ proves the upper bound.
\end{proof}

For the rest of this subsection, fix $\lambda_2$ as in the proof of
\Cref{prop:polar-estimate}, with the notation $L,u,U$ introduced there.

\begin{proposition}
\label{prop:plane-estimates}
With $\mathcal C$ and $m_{\boldsymbol n,p}$ as in
\Cref{not:principalization}, one has
\begin{align}
 m_X&=\sum_{p\in\mathcal C}m_{\boldsymbol n,p}^2,
 \label{eq:resolution-multiplicity}\\
 M_{\boldsymbol n}&=\sum_{p\in\mathcal C}m_{\boldsymbol n,p}.
 \label{eq:resolution-sums}
\end{align}
In particular,
\begin{equation}
 m_{\boldsymbol{n}}\leq M_{\boldsymbol{n}},\qquad
 m_X\leq m_{\boldsymbol{n}}M_{\boldsymbol{n}}.
 \label{eq:pointbounds}
\end{equation}
\end{proposition}

\begin{proof}
For each branch $B$ of $V$, let $\gamma_B$ be a
parametrization  and let $B_p$ be the strict transform of $B$
at $p\in\mathcal C$.
Since $L$ is not tangent to any branch of $W$, the linear map $\lambda_2$
does not vanish on the first nonzero vector coefficient of
$\boldsymbol n\circ\gamma_B$. Therefore, using
\cite[Lemma~11.2.1(ii)]{Wall}, we obtain
\[
 \ord_0(u\circ\gamma_B)
 =\min_{h\in I}\ord_0(h\circ\gamma_B)
 =\sum_{p\in\mathcal C}\mult_p(B_p)m_{\boldsymbol n,p}.
\]
Summing over the branches of $V$ gives
\[
 i_0(v,u)
 =\sum_B\ord_0(u\circ\gamma_B)\\
 =\sum_{p\in\mathcal C}m_{\boldsymbol n,p}
       \sum_B\mult_p(B_p)
 =\sum_{p\in\mathcal C}m_{\boldsymbol n,p}^2.
\]
The last equality follows from
(\ref{cond:casas-alvero}). Since $i_0(v,u)=m_X$ by the proof of
\Cref{prop:polar-estimate}, this proves
\eqref{eq:resolution-multiplicity}. By (\ref{cond:casas-alvero}) and \cite[Theorem~6.5.9]{Wall}, we have
$
 \mu(V)=\sum_{p\in\mathcal C}m_{\boldsymbol n,p}(m_{\boldsymbol n,p}-1)-b+1.
$
Together with Milnor's formula
$
 \mu(V)=2\delta(V,0)-b+1
$, we get
\[
 2\delta(V,0)=\sum_{p\in\mathcal C}m_{\boldsymbol n,p}(m_{\boldsymbol n,p}-1).
\]
Consequently, by \eqref{eq:resolution-multiplicity} and the definition
of $M_{\boldsymbol n}$ in \eqref{eq:section-invariants}, we obtain
\eqref{eq:resolution-sums}. Since $m_{\boldsymbol n,0}=m_{\boldsymbol n}$,
\eqref{eq:resolution-sums} gives the first inequality in
\eqref{eq:pointbounds}. Moreover, the multiplicity of
strict transforms does not increase by blowups, which proves the second inequality in \eqref{eq:pointbounds}.
\end{proof}

\begin{remark}
\label{rem:alternative-resolution}
We explain here that the inequalities \eqref{eq:pointbounds} could have also been obtained by just
resolving the general member $V$ (without necessarily principalizing
the ideal $I$). An embedded  resolution of $V$ may not blow up some of the base
points of $I$ at which the strict transform of $V$ is already
smooth. Observe that these points contribute to the sum in
\eqref{eq:resolution-multiplicity}, but not to the formula for
$\delta(V,0)$. Denote by $\mathcal S\subset\mathcal C$
the set of centers of the embedded resolution of $V$. Then, one would need to
choose $\lambda_2$ also sufficiently general so that the strict
transforms of $U$ and $V$ have the same multiplicity
$m_{\boldsymbol n,p}$ at all points in $p\in\mathcal S$. Then, 
Noether's formula \cite[Theorem~3.3.1]{CasasAlvero} yields
$m_X=\sum_{p\in\mathcal S}m_{\boldsymbol n,p}^2+R$, where
$R\geq0$ is the sum of the remaining intersection multiplicities of
the strict transforms of $U$ and $V$ above the origin. The formula for the delta
invariant would also give, in this case,
$M_{\boldsymbol n}=\sum_{p\in\mathcal S}m_{\boldsymbol n,p}+R$.
The residual term contributes $R$ to $M_{\boldsymbol n}$ and
$(m_{\boldsymbol n}-1)R$ to
$m_{\boldsymbol n}M_{\boldsymbol n}-m_X$, so its nonnegativity
suffices to recover \eqref{eq:pointbounds} and the subsequent
arguments. We choose principalization over this method because it gives the exact identities
\eqref{eq:resolution-multiplicity} and \eqref{eq:resolution-sums} which we think interesting in their own. Also they allow us to impose a simpler condition
on $\lambda_2$  and the residual term $R$ does not appear.
\end{remark}

\section{Knot cobordisms}
\label{sec:knot-cobordisms}

For the ball $\mathbb{B}$ fixed in \Cref{sec:plane-sections}, the intersection
$
 \mathbb{B}_W=\mathbb{B}\cap H_0
$
is a Milnor ball for $(W,0)$. Set
\[
 \mathbb{S}_W=\partial \mathbb{B}_W,\qquad \Lambda_W=W\cap \mathbb{S}_W.
\]
So $\Lambda_W$ is a smooth oriented link in $\mathbb{S}_W$. Choose an
orientation preserving diffeomorphism $\mathbb{S}_W\cong \mathbb{S}^3$. Via this identification  we can look at the positive torus knot
$T(m_X,m_X+1)$ as a knot in $\mathbb{S}_W$. 

An orientation on a manifold induces an orientation on its boundary  by the
following convention: an outward pointing normal vector followed by a
positive boundary basis gives a positive basis of the manifold. If $Y$ is an oriented manifold, we denote by $-Y$ the manifold with the other orientation.

\begin{definition}
\label{def:cobordisms}
Let $Y_0$ and $Y_1$ be closed oriented smooth manifolds of dimension $d$. A {\em cobordism} from
$Y_0$ to $Y_1$ is a compact connected smooth manifold $Z$ of eimension $(d+1)$ with
a chosen orientation and a fixed identification
\[
 \partial Z=-Y_0\sqcup Y_1
\]
of its boundary with the disjoint union of $-Y_0$ and $Y_1$. Here $\partial Z$ carries the orientation
induced by $Z$ with our convention. The boundary pieces $-Y_0$ and $Y_1$
are called the {\em incoming} and {\em outgoing} boundaries,
respectively.
\end{definition}

\begin{definition}
\label{def:link-cobordisms}
Let $Y_0,Y_1$ be closed oriented smooth $3$-manifolds, and let
$L_i\subset Y_i$ be oriented links for $i=0,1$. An {\em oriented
link cobordism} from $(Y_0,L_0)$ to $(Y_1,L_1)$ is a pair
$(Z,F)$, where $Z$ is a cobordism from $Y_0$ to $Y_1$ and
$F$ is a cobordism from $L_0$ to $L_1$ which is, moreover, a smooth oriented
surface properly embedded in $Z$.
The links $L_0$ and $L_1$ are called the {\em incoming and outgoing links},
respectively. We denote the genus of $F$ by $g(F)$.
\end{definition}

\begin{remark}
\label{rem:boundary-orientations}
For a closed oriented $3$-manifold $Y$ and an oriented link
$L\subset Y$, we orient $Y\times[0,1]$ and $L\times[0,1]$ so that
their induced boundary pairs are $(-Y,-L)$ at level $0$ and
$(Y,L)$ at level $1$.
Note that these orientations are opposite to the {\em natural} product orientations in
the displayed order.
\end{remark}

Given an oriented link $L\subset \mathbb{S}^3$, we denote by $g_3(L)$ its
{\em Seifert genus}, that is, the smallest genus realized by its connected Seifert surfaces.
Every fiber surface $F$ for $L$ realizes this genus
\cite[Proposition~2.19(3)]{Rudolph}
(Proposition~19(3) in arXiv:math/0411115v2), so $g_3(L)=g(F)$.
In particular, all fiber surfaces for $L$ have the same Euler
characteristic and so they are homeomorphic. Furthermore, the same proposition also shows that any two
such surfaces are actually isotopic.

We summarize the rest of the section. We construct two oriented link cobordisms whose composition has a knot
on each end. For $s_0\neq0$ small enough, let $K_{Q_j}$
be the knot of the transverse plane branch $(W_{s_0},Q_j)$, and
let $K_{\mathrm{tr}}$ be their oriented connected sum, defined in
\eqref{eq:transverse-knot}. The exterior surface $F_W$ joins
$K_{\mathrm{tr}}$ to the link $\Lambda_{s_0}$. The fixed
boundary trivialization in \Cref{subsec:exterior-cobordism}
identifies this outgoing link with $\Lambda_W$, which has $b$
components. The braid cobordism $C$ joins $\Lambda_W$ to
$T(m_X,m_X+1)$.
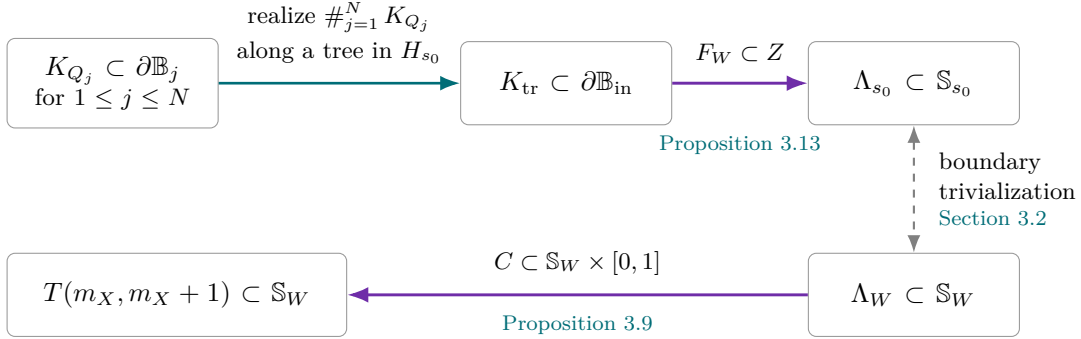
\begin{figure}[H]
\centering
\begin{tikzpicture}[
  >=Latex,
  every node/.style={font=\small,align=center},
  object/.style={
    draw=black!40,rounded corners=3pt,fill=white,
    inner xsep=7pt,inner ysep=7pt,minimum height=11mm
  },
  construction/.style={->,draw=DarkTurquoise,line width=1pt},
  cobordism/.style={->,draw=ReferenceViolet,line width=1.1pt},
  identification/.style={
    <->,draw=black!50,line width=.8pt,dash pattern=on 3pt off 3pt
  },
  explanation/.style={font=\footnotesize,inner sep=2pt},
  reference/.style={font=\scriptsize,inner sep=2pt}
]
\node[object,text width=2.3cm] (local) at (1.6,0)
  {$K_{Q_j}\subset\partial \mathbb{B}_j$\\[-1pt]
   {\footnotesize for $1\leq j\leq N$}};
\node[object,text width=2.3cm] (transverse) at (7.6,0)
  {$K_{\mathrm{tr}}\subset\partial \mathbb{B}_{\mathrm{in}}$};
\node[object,text width=2.3cm] (moving) at (12.2,0)
  {$\Lambda_{s_0}\subset \mathbb{S}_{s_0}$};
\node[object,text width=2.3cm] (central) at (12.2,-2.8)
  {$\Lambda_W\subset \mathbb{S}_W$};
\node[object,text width=4cm] (torus) at (2.45,-2.8)
  {$T(m_X,m_X+1)\subset \mathbb{S}_W$};

\draw[construction] (local.east) --
  node[explanation,above=5pt]
    {realize $\mathop{\#}_{j=1}^{N}K_{Q_j}$\\[2pt]
     along a tree in $H_{s_0}$}
  (transverse.west);
\draw[cobordism] (transverse.east) --
  node[explanation,above=5pt] {$F_W\subset Z$}
  node[reference,below=17pt] {\Cref{prop:exterior}}
  (moving.west);
\draw[identification] (moving.south) --
  node[explanation,right=7pt,align=left]
    {boundary\\trivialization\\[-1pt]
     {\scriptsize\Cref{subsec:exterior-cobordism}}}
  (central.north);
\draw[cobordism] (central.west) --
  node[explanation,above=5pt] {$C\subset \mathbb{S}_W\times[0,1]$}
  node[reference,below=5pt] {\Cref{prop:braid}}
  (torus.east);
\end{tikzpicture}
\caption{Gluing $F_W$ and $C$ gives a cobordism from
$K_{\mathrm{tr}}$ to $T(m_X,m_X+1)$.}
\label{fig:knot-cobordisms}
\end{figure}

If we glue $F_W$ to $C$ along these $b$ circles, we get a connected
knot cobordism inside a manifold diffeomorphic to $\mathbb{S}^3\times[0,1]$. We compute its
genus in \Cref{cor:genus-identities}. In
\Cref{sec:numerical-estimates}, we apply the estimate coming from knot Floer homology
to obtain \Cref{cor:upsilon-comparison}.

We recall the oriented band attachment operation, see for example,
\cite[Definition~12.3.3]{Kawauchi}.

\begin{definition}
\label{def:band-attachment}
Let $I=[0,1]$ denote the unit interval. A {\em band attached} to an
oriented link $L$ in an oriented $3$-manifold $Y$ is a smooth
embedding
$
 a:I\times I\longrightarrow Y
$
such that
$
 a(I\times I)\cap L=a(\{0,1\}\times I).
$
{\em Attaching the band} means replacing these two attaching arcs by
$a(I\times\{0,1\})$ and smoothing the four corners. The band is
oriented if the orientation of $L$ along the attaching arcs agrees
with the boundary orientation induced by an orientation of the band.
\end{definition}

\begin{remark}
\label{rem:braid-letter-cobordism}
As explained, for example, in \cite[Section~13.1, p.~172]{Kawauchi},
a sequence of $s$ oriented band attachments from $L$ to $L'$ can be realized
by a compact smooth oriented properly embedded surface
$F\subset Y\times[0,1]$, with boundary
$\partial F = -(L\times\{0\})\sqcup(L'\times\{1\})$. Moreover, we can assume that the surface $F$ is a
product cylinder near levels $0$ and $1$. 
The projection
\[
 q_F\colon F\longrightarrow[0,1],\qquad (x,t)\longmapsto t,
\]
has exactly $s$ critical points, all saddles
(see also \Cref{fig:band-attachment-steps}). Each band lowers the Euler
characteristic by one, so $\chi(F)=-s$. Every connected component of $F$ meets $L'\times\{1\}$, because otherwise
$q_F$ would have maximum on a point in the interior on that component.
Thus $F$ is connected if $L'$ is a knot.

Deleting a letter from a positive braid word amounts to replacing the corresponding
crossing of its closure by its oriented smoothing. This  is equivalent to an oriented
band attachment, as illustrated in \Cref{fig:word-surface}.
Inserting a letter also corresponds to an oriented band
attachment, so both operations are realized by the same
construction in $\mathbb{S}^3\times[0,1]$.
\end{remark}
\begin{figure}[!ht]
\centering
\begingroup
\footnotesize
\graphicspath{{./}}
\def\svgwidth{\linewidth}
\input{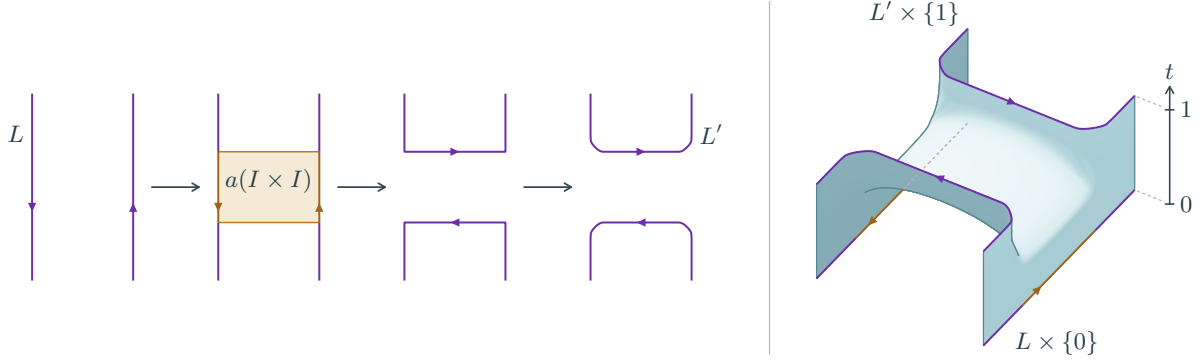}
\par
\endgroup
\caption{An oriented band attachment from $L$ to $L'$ on the left, and its realization by a saddle on the right. Outside a neighborhood of the band, the cobordism is the product of the unchanged part of the link with $[0,1]$.}
\label{fig:band-attachment-steps}
\end{figure}

\FloatBarrier

\subsection{Braid construction and cobordism with the torus knot}
\label{subsec:braid-cobordism}

In this subsection, we construct a link cobordism from $\Lambda_W$ to the torus knot
$T(m_X,m_X+1)$ by deleting and inserting positive braid letters.

For any positive braid word $\mathbf a$, let
$\ell(\mathbf a)$ denote its number of letters.
Its {\em word surface} is obtained from one disk for each strand by attaching
one band with a positive half twist for each letter
(see \Cref{fig:word-surface}). If the word contains every Artin
generator, these bands connect all the disks, and Stallings' theorem
implies that the word surface is a fiber surface
\cite[Theorem~2]{Stallings}.
\begin{figure}[htbp]
\centering
\begingroup
\footnotesize
\graphicspath{{./}}
\def\svgwidth{\linewidth}
\input{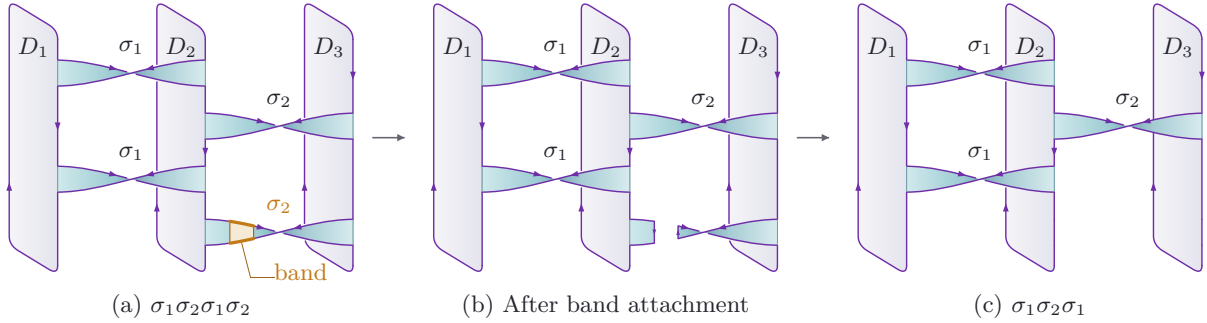}
\par
\endgroup
\caption{The operation of attaching the band (yellow) in (a) is equivalent to deleting the corresponding $\sigma_2$ from the word.
The result gives (c) after an isotopy.}
\label{fig:word-surface}
\end{figure}

The construction in the proof of the next proposition adapts to links the band attachment argument in \cite[proof of Lemma~3.5]{FellerKrcatovich}. The result could also be obtained by first inserting $b-1$ positive braid generators to join the $b$ components into a knot, and then applying the construction from the cited paper. We give the details here for the case of links for completeness. 
\begin{proposition}\label{prop:braid}
There is a connected smooth oriented link cobordism
$C\subset \mathbb{S}_W\times[0,1]$ from $\Lambda_W$ to
$T(m_X,m_X+1)$ such that
\[
 g(C)=\delta(W,0)-\frac{m_X(m_X-1)}2.
\]
\end{proposition}

\begin{proof}
By \eqref{eq:section-invariants}, $W$ has multiplicity $m_X$ and so its link $\Lambda_W$ can be represented by a closed braid
on $m_X$ strands. We recall this construction from
\cite[Section~8.3]{BrieskornKnorrer}.
Choose linear coordinates $(\xi,\zeta)$ on
$H_0=\{\lambda_1=0\}$ such that $\{\xi=0\}$ is not a tangent line
of $W$. Then $\xi|_W$ is finite of degree
\[
 i_0(W,\{\xi=0\})=\mult_0(W)=m_X.
\]
After rescaling $\zeta$, choose a bidisk of radius
$\rho>0$ small enough that $W$ meets its boundary
only in the interior of the solid torus
$
 \mathcal T=\{|\xi|=\rho,\ |\zeta|\leq\rho\}.
$
The restriction
$
 \xi\colon W\cap\mathcal T\longrightarrow\{|\xi|=\rho\}
$
is an orientation preserving covering of degree $m_X$, so
$W\cap\mathcal T$ is the closure $\widehat\beta$ of a braid
$\beta\in B_{m_X}$.
Moreover, by \cite[Section~8.3, Corollary~7]{BrieskornKnorrer},
$\widehat\beta$ and $\Lambda_W$ have the same oriented link type.
The \emph{full twist} on $m_X$ strands is the braid
\[
 \Delta_{m_X}^2=(\sigma_1\cdots\sigma_{m_X-1})^{m_X},
\]
obtained by rotating the configuration of $m_X$ points in the disk
one full positive turn as one moves along the braid.
By \cite[Lemma~13.5 and its proof]{Fomin}, we may choose $\beta$
so that $\Delta_{m_X}^{-2}\beta$ is positive.
For this choice of $\beta$, fix an orientation preserving diffeomorphism of
pairs $(\mathbb{S}^3,\widehat\beta)\cong(\mathbb{S}_W,\Lambda_W)$.
Let $\mathbf d$ be the displayed positive word representing
$\Delta_{m_X}^2$, and let $\mathbf w$ be a positive word representing
$\Delta_{m_X}^{-2}\beta$. Then
\[
 \mathbf a=\mathbf d\mathbf w
\]
is a positive word representing $\beta$.
The cycles of the permutation associated with $\beta$ correspond to the
$b$ connected components of $\Lambda_W$. Since the full twist induces the identity
permutation, the permutation of the braid represented by $\mathbf w$
also has $b$ cycles.
Let $r$ be the number of different Artin generators that appear in $\mathbf w$. Form the
graph with vertices $1,\ldots,m_X$ and an edge $\{i,i+1\}$
for each generator $\sigma_i$ occurring in $\mathbf w$.
This graph is obtained from the chain $1-2-\cdots-m_X$ by deleting
$m_X-1-r$ edges, so it has $m_X-r$ connected components. Every generator that appears in
$\mathbf w$, preserves these sets of vertices, so every cycle of the
permutation defined by $\mathbf w$ lies in one of them. Therefore,
\[
 b\geq m_X-r \quad \text{ and } \quad r\geq m_X-b.
\]
Choose $m_X-b$ different generators appearing in the word $\mathbf w$,
and retain one occurrence of each in the order in which they appear.
The resulting subword $\boldsymbol\gamma_0$ has length $m_X-b$.
Let $\boldsymbol\eta$ be any word containing exactly once and  in any order, each of the
$b-1$ generators (among $\sigma_1,\ldots,\sigma_{m_X-1}$) that do not
appear in $\boldsymbol\gamma_0$. Then the word
\[
 \boldsymbol\gamma=\boldsymbol\gamma_0\boldsymbol\eta
\]
has every Artin generator appearing exactly once. The words
$\boldsymbol\eta$ and $\boldsymbol\gamma_0$ are empty when $b=1$ and
$b=m_X$, respectively.  Humphreys' conjugacy argument for Coxeter elements
\cite[\S~3.16]{HumphreysCoxeter} applies here and shows that the braid represented by
$\boldsymbol\gamma$ is conjugate to $c=\sigma_1\cdots\sigma_{m_X-1}$.
Since the full twist $\Delta_{m_X}^2=c^{m_X}$ is central, the braid
represented by $\mathbf d\boldsymbol\gamma$ is conjugate to
$c^{m_X+1}$. Its closure is therefore isotopic to $T(m_X,m_X+1)$. In the example of \Cref{fig:braid-movie},
$m_X=4$, $\mathbf w=\sigma_1^3\sigma_3^3$, and $b=2$, so we choose
$\boldsymbol\gamma_0=\sigma_1\sigma_3$ and $\boldsymbol\eta=\sigma_2$. 

Now, we construct the cobordism by deleting and inserting letters.
By \Cref{rem:braid-letter-cobordism}, the following letter operations can be
realized by oriented band attachments on the braid closures. 
\[
 \mathbf d\mathbf w
 \longrightarrow\mathbf d\boldsymbol\gamma_0
 \longrightarrow\mathbf d\boldsymbol\gamma_0\boldsymbol\eta.
\]
We recall that we view these band attachments in
$\mathbb S_W$ via the identification fixed at the beginning of this proof. The first arrow deletes the letters of $\mathbf w$ not retained in
$\boldsymbol\gamma_0$. The second arrow represents the following operation: for each successive letter
$\sigma_i$ of $\boldsymbol\eta$, attach, below the current braid, the
standard oriented band with one positive half twist between the strands
occupying positions $i$ and $i+1$. These $b-1$ insertions are illustrated
in \Cref{fig:braid-movie}. Let $K$ be the closure of the braid
represented by $\mathbf d\boldsymbol\gamma$. And let
$C_0\subset\mathbb S_W\times[0,1]$ be the surface traced by these band
attachments, from $\Lambda_W\times\{0\}$ to $K\times\{1\}$.
Since $K$ is a knot, \Cref{rem:braid-letter-cobordism} shows that $C_0$
is connected.
The closures of conjugate braids are ambient isotopic
\cite[Theorem~1.2.5]{Kawauchi}. So there is an
ambient isotopy
$h_t:\mathbb S_W\to\mathbb S_W$, $t\in[0,1]$, stationary near the endpoints,
with $h_0=\mathrm{id}$ and $h_1(K)=T(m_X,m_X+1)$ as oriented knots.
The isotopy defines the annulus
\[
 A=\{(h_t(x),1+t):x\in K,\ t\in[0,1]\}
 \subset\mathbb S_W\times[1,2].
\]
Glue $A$ to $C_0$ along their common boundary $K\times\{1\}$. After rescaling $[0,2]$ to
$[0,1]$, denote this connected oriented link cobordism by $C$.

Let $S_{\mathbf a}\subset\mathbb S^3$ denote the word surface of
$\mathbf a=\mathbf d\mathbf w$. We see $S_{\mathbf a}$ in $\mathbb S_W$ by the identification $\mathbb S^3 \to \mathbb S_W$ fixed at the beginning of this proof.
The word $\mathbf d$ contains every Artin generator and $\mathbf a$ is positive,
Stallings' theorem \cite[Theorem~2]{Stallings} shows that $S_{\mathbf a}$
is a connected fiber surface for $\Lambda_W$.
The surface $S_{\mathbf a}$ and a page of the Milnor fibration of $(W,0)$
are therefore fiber surfaces for the same oriented link.
By \cite[Proposition~2.19(3)]{Rudolph}, these surfaces are isotopic and
have the same Euler characteristic.
The surface $S_{\mathbf a}$ is constructed from $m_X$ disks by attaching
$\ell(\mathbf d)+\ell(\mathbf w)$ bands. Since
$\ell(\mathbf d)=m_X(m_X-1)$, we obtain
\[
 \begin{aligned}
  \chi(S_{\mathbf a})
  &=m_X-\ell(\mathbf d)-\ell(\mathbf w)\\
  &=m_X-m_X(m_X-1)-\ell(\mathbf w)\\
  &=1-\mu(W)=b-2\delta(W,0),
 \end{aligned}
\]
where the last equality follows from Milnor's formula. On the other hand, the construction of $C_0$ uses $\ell(\mathbf w)-(m_X-b)$ bands for
the deletions of letters, and $b-1$ bands for the insertions. That is a total of
\[
 s=\ell(\mathbf w)-(m_X-b)+(b-1)
\]
bands. So $\chi(C_0)=-s$. Gluing the annulus $A$
preserves the Euler characteristic, hence $\chi(C)=\chi(C_0)=-s$.
Since $C$ is connected and has $b+1$ boundary components, we obtain
\[
 2g(C)=s-b+1
 =\ell(\mathbf w)-(m_X-b)
 =2\delta(W,0)-m_X(m_X-1).
\]
Dividing by $2$ proves the formula in the statement.
\end{proof}
\begin{figure}[!ht]
\centering
\begingroup
\footnotesize
\graphicspath{{./}}
\def\svgwidth{\linewidth}
\input{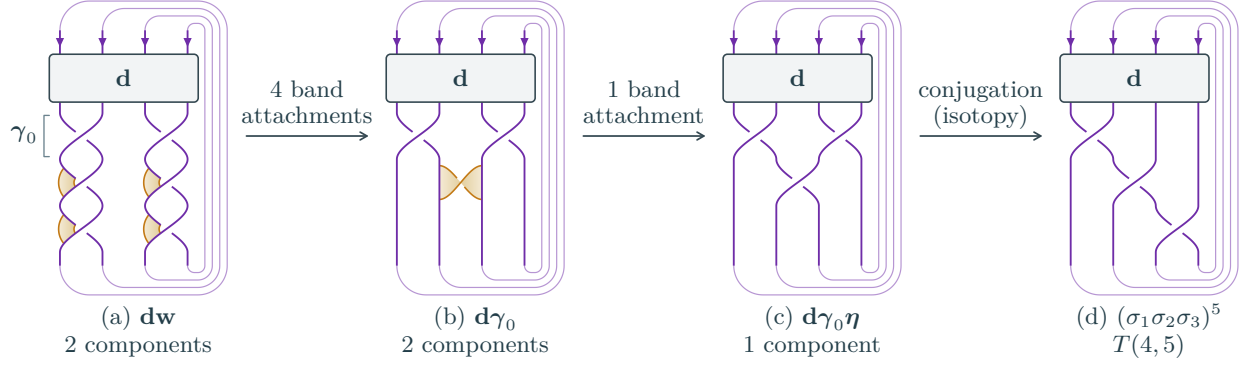}
\par
\endgroup
\caption{Example of the strategy used in the proof of
\Cref{prop:braid} for $m_X=4$. Consider the word $\mathbf w=\sigma_1^3\sigma_3^3$
for this example. Each diagram is read from top to bottom, in the direction of the
strands, and all crossings are positive.  The box stands for the word
$\mathbf d=(\sigma_1\sigma_2\sigma_3)^4$ representing $\Delta_4^2$. In
(a), (b) and (c), the  letters $\sigma_1$ and $\sigma_3$ are
drawn at the same heights. In (a), the bracket marks
$\boldsymbol\gamma_0=\sigma_1\sigma_3$, and attaching the four pale yellow
bands deletes the other letters of $\mathbf w$. In (b), attaching the
band with a positive half twist $\boldsymbol\eta=\sigma_2$.
Since $\sigma_1\sigma_3=\sigma_3\sigma_1$ and $\mathbf d$ represents a
central braid, the word
$\sigma_3^{-1}\mathbf d\boldsymbol\gamma_0\boldsymbol\eta\,\sigma_3$
represents the same braid as $(\sigma_1\sigma_2\sigma_3)^5$, whose
closure is $T(4,5)$. Here $b=2$, and the cobordism $C$ obtained from the
$s=5$ band attachments and the isotopy from (c) to (d) has $3$ boundary
components. So $\chi(C)=-5$ and $g(C)=2$.}
\label{fig:braid-movie}
\end{figure}
\subsection{The exterior cobordism}
\label{subsec:exterior-cobordism}

For the ball $\mathbb{B}$ and $\eta$ fixed in
\Cref{subsec:general-plane-sections}, we set
\[
 \Delta_\eta=\{s\in\C:|s|<\eta\},\qquad
 \overline\Delta_\eta=\{s\in\C:|s|\leq\eta\},\qquad
 \mathbb{B}_{\mathrm{out},s}=\mathbb{B}\cap H_s.
\]
In particular, each $\mathbb{B}_{\mathrm{out},s}$ is a closed Euclidean $4$-ball. Denote $\mathbb{S}_s=\partial \mathbb{B}_{\mathrm{out},s}$ and $
 \Lambda_s=W_s\cap \mathbb{S}_s.
$
By \Cref{subsec:general-plane-sections}, $W_s\cap\partial\mathbb B$
lies in $X_{\mathrm{reg}}$, and $H_s$ meets
$X_{\mathrm{reg}}\cap\partial\mathbb B$ transversely for $|s|\leq\eta$.
After decreasing $\eta$ if necessary, $H_s$ also meets
$\partial\mathbb B$ transversely.
Ehresmann's lemma for pairs therefore gives an orientation preserving trivialization
\[
 \bigl(\partial\mathbb B\cap\lambda_1^{-1}(\overline\Delta_\eta),
 X\cap\partial\mathbb B\cap\lambda_1^{-1}(\overline\Delta_\eta)\bigr)
 \cong (\mathbb S_0,\Lambda_0)\times\overline\Delta_\eta
\]
over the disk $\overline\Delta_\eta$, where the map to
the base on the left is $\lambda_1$. In particular, this gives
orientation preserving identifications
\begin{equation}
 \label{eq:boundary-identification}
 (\mathbb S_s,\Lambda_s)\cong(\mathbb S_0,\Lambda_0)
 =(\mathbb S_W,\Lambda_W)
 \qquad\text{for }s\in\overline\Delta_\eta.
\end{equation}
We fix such a trivialization. After maybe decreasing $\eta$, the conclusions
of \Cref{prop:plane-sections} hold for $0<|s|\leq\eta$ and the points $Q_j(s)$ lie in the interior of
$\mathbb{B}_{\mathrm{out},s}$.
Choose $s_0\in\Delta_\eta \setminus \{0\}$ small enough. By
(\ref{cond:isolated-critical-point}), $s_0$ is a regular value of $v$. Put
\[
 Q_j=Q_j(s_0)\qquad\text{for }1\leq j\leq N.
\]
For each $j$, choose an Euclidean Milnor ball $\mathbb{B}_j$ centered at
$Q_j$, with the balls pairwise disjoint and contained in
the interior of $\mathbb{B}_{\mathrm{out},s_0}$, and set
\[
 K_{Q_j}=W_{s_0}\cap\partial \mathbb{B}_j.
\]
After choosing an orientation preserving identification
$\partial \mathbb{B}_j\cong \mathbb{S}^3$, we see $K_{Q_j}$ as an oriented knot in
$\mathbb{S}^3$. It is a knot because $(W_{s_0},Q_j)$ is unibranch
(recall \Cref{prop:plane-sections}).
Let $K_{\mathrm{tr}}\subset \mathbb{S}^3$ be a representative of the
oriented connected-sum knot type
\begin{equation}
 K_{\mathrm{tr}}
 \simeq K_{Q_1}\mathbin{\#}\cdots\mathbin{\#}K_{Q_N}.
 \label{eq:transverse-knot}
\end{equation}
The subscript ``$\mathrm{tr}$'' refers to the
{\em transverse} plane curve germs $(W_{s_0},Q_j)$.

In the next proposition, we construct a link cobordism from $K_{\mathrm{tr}}$ to
$\Lambda_{s_0}$, realized as a subsurface of $W_{s_0}$ inside a
$4$-manifold diffeomorphic to $\mathbb S^3\times[0,1]$.
This manifold is obtained by removing from $\mathbb B_{\mathrm{out},s_0}$
the interior of a ball formed by joining the local Milnor balls
along the thickening of an embedded tree.
For similar constructions, see \cite[proof of Lemma~2.1]{BaaderFeller}
and \cite[proof of Lemma~2.3 in the arXiv version]{BorodzikLivingston}.

\begin{proposition}
\label{prop:exterior}
There is a connected smooth oriented link cobordism $(Z,F_W)$
from $(\mathbb{S}^3,K_{\mathrm{tr}})$ to
$(\mathbb{S}_{s_0},\Lambda_{s_0})$ such that $Z$ is orientation preserving
diffeomorphic to $\mathbb{S}^3\times[0,1]$ and
\[
 \chi(F_W)=-\mu(V).
\]
\end{proposition}

\begin{proof}
By \Cref{rem:milnor-fibers}, $V_{s_0}$ is connected
and $\chi(V_{s_0})=1-\mu(V,0)$.
By \Cref{prop:plane-sections}, the  preimages
\[
 D_j=\boldsymbol n_{s_0}^{-1}(W_{s_0}\cap\mathbb B_j),
 \qquad\text{for }1\leq j\leq N,
\]
are pairwise disjoint disks in the interior of $V_{s_0}$.
Set $I=[0,1]$. Removing the interiors of the $D_j$ does not
disconnect the surface, so we can choose smooth embedded arcs
$\alpha_i:I\to W_{s_0}$, for $1\leq i\leq N-1$, with pairwise
disjoint images in the smooth locus, such that the lifted arcs
$\boldsymbol n_{s_0}^{-1}(\alpha_i(I))$ meet the circles
$\partial D_j$ transversely at their endpoints and have interiors contained in
$\operatorname{int}(V_{s_0})\setminus\bigcup_jD_j$.
Moreover, we choose the arcs so that their incidence graph with the disks $D_j$
is a tree.
Let $D^3$ be the closed unit disk in $\mathbb R^3$. Identify
$[-1,1]$ with the diameter $[-1,1]\times\{(0,0)\}\subset D^3$ for this construction.
Choose smooth embeddings
\begin{equation}
 \widehat\alpha_i:I\times D^3\longrightarrow
 \operatorname{int}(\mathbb B_{\mathrm{out},s_0}),
 \label{eq:product-neighborhoods}
\end{equation}
with pairwise disjoint images, such that for each $i$ the following holds:
\begin{enumerate}
\setlength{\itemsep}{2pt}
\setlength{\parsep}{0pt}
\renewcommand{\labelenumi}{(\roman{enumi})}
\item $\widehat\alpha_i(t,0)=\alpha_i(t)$ for every $t\in I$.
\item The intersection with $W_{s_0}$ is the band 
$ \widehat\alpha_i(I\times D^3)\cap W_{s_0}
 =\widehat\alpha_i(I\times[-1,1]).
$
And the restriction of $\widehat\alpha_i$ to $I\times[-1,1]$
preserves the orientation of $W_{s_0}$.
\item The intersection with the balls consists of the two boundary disks, that is,
\[
 \widehat\alpha_i(I\times D^3)\cap\bigcup_{j=1}^N\mathbb B_j
 =\widehat\alpha_i(\partial I\times D^3).
\]
\end{enumerate}
Let $\mathbb B_{\mathrm{in}}$ be the union of the $\mathbb B_j$ and
the images $\widehat\alpha_i(I\times D^3)$ and set
$
 D_{\mathrm{in}}=W_{s_0}\cap\mathbb B_{\mathrm{in}}.
$
Since the incidence graph is a tree, the ball $\mathbb B_{\mathrm{in}}$ is a
boundary connected sum of the balls $\mathbb B_j$, and
$\boldsymbol n_{s_0}^{-1}(D_{\mathrm{in}})$ is a boundary connected sum
of the disks $D_j$. Since $\boldsymbol n_{s_0}$ is a homeomorphism,
$D_{\mathrm{in}}$ is a topological disk with boundary corners.
See \Cref{fig:normalization-tree-surface}.
\begin{figure}[!ht]
\centering
\begingroup
\footnotesize
\graphicspath{{./}}
\def\svgwidth{\linewidth}
\input{normalization_tree_surface_labels.tex}
\par
\endgroup
\caption{In yellow, the disk
$\boldsymbol n_{s_0}^{-1}(D_{\mathrm{in}})\subset V_{s_0}$ (left) and its image $D_{\mathrm{in}}\subset W_{s_0}$
(right). Here  $N=4$.}
\label{fig:normalization-tree-surface}
\end{figure}

Now we explain how to round the boundary corners of $D_{\mathrm{in}}$ in
$W_{s_0}$. Fix a corner point
$p\in W_{s_0}\cap\widehat\alpha_i(\partial I\times\partial D^3)$.
Using the product neighborhoods \eqref{eq:product-neighborhoods}, choose real ambient coordinates
$(x,y,z_1,z_2)$ centered at $p$ in which $W_{s_0}$ is locally given
by $z_1=z_2=0$ and $\mathbb B_{\mathrm{in}}$ by $x\leq0$ or $y\leq0$.
Fix  $a>0$ small enough and define the auxiliary smooth function
$h:\mathbb R\to\mathbb R$ to construct a parametrization
$\kappa_a:I\to\mathbb R^2$  by
\[
 h(t)=
 \begin{cases}
  0,&t\leq0,\\
  e^{1-1/t},&t>0,
 \end{cases}
 \qquad
 \kappa_a(t)=a\bigl(h(t),h(1-t)\bigr).
\]
The map $\kappa_a$ has nonvanishing derivative
$\kappa_a'(t)=a\bigl(h'(t),-h'(1-t)\bigr)$ and parametrizes a smooth
arc from $(0,a)$ to $(a,0)$ that (since all derivatives of $h$ vanish
at $0$) joins the two axes smoothly and replaces the corner at $p$
(see \Cref{fig:corner-rounding}).
Using the product structure in \eqref{eq:product-neighborhoods},
we extend the rounding in the $(z_1,z_2)$ directions to round
$\mathbb B_{\mathrm{in}}$, whose boundary is then transverse to $W_{s_0}$.
We retain the notation $\mathbb B_{\mathrm{in}}$ and $D_{\mathrm{in}}$
after rounding. Then $\mathbb B_{\mathrm{in}}$ is a smooth $4$-ball in
$\operatorname{int}(\mathbb B_{\mathrm{out},s_0})$, and
$\boldsymbol n_{s_0}^{-1}(D_{\mathrm{in}})$ is a smooth disk in
$\operatorname{int}(V_{s_0})$.
\begin{figure}[!ht]
\centering
\begingroup
\footnotesize
\graphicspath{{./}}
\def\svgwidth{0.95\linewidth}
\input{corner_rounding_labels.tex}
\par
\endgroup
\caption{Rounding of a corner at the level of $D_\mathrm{in}$}
\label{fig:corner-rounding}
\end{figure}

\begin{figure}[!ht]
\centering
\begingroup
\footnotesize
\graphicspath{{./}}
\def\svgwidth{0.95\linewidth}
\input{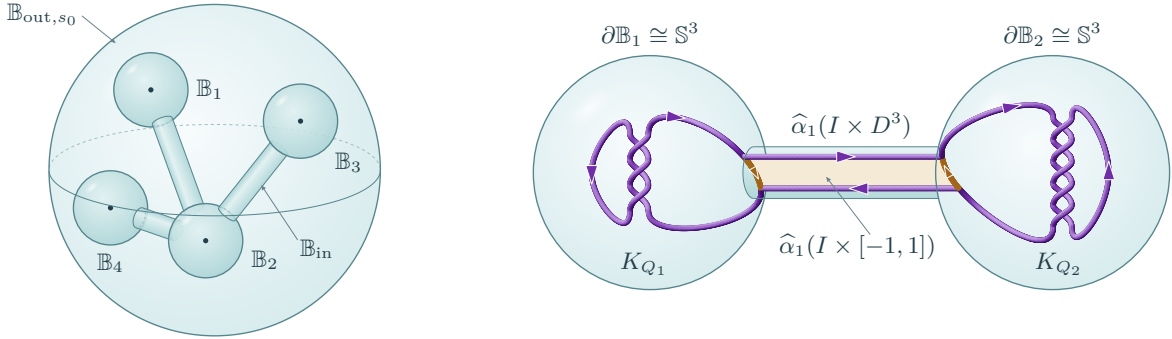}
\par
\endgroup
\caption{On the left we see
$\mathbb B_{\mathrm{in}}\subset\mathbb B_{\mathrm{out},s_0}$. And on the right we see an schematic picture of one handle realizing $K_{Q_1}\mathbin{\#}K_{Q_2}$.}
\label{fig:inner-tree-neighborhood}
\end{figure}

Each handle realizes the oriented connected sum on the boundary
(see \Cref{rem:tree-connected-sum}), giving an orientation
preserving diffeomorphism $
(\partial\mathbb B_{\mathrm{in}},
\partial D_{\mathrm{in}})
\cong(\mathbb S^3,K_{\mathrm{tr}})
$ (see \Cref{fig:inner-tree-neighborhood}). Set
\[
 Z=\mathbb B_{\mathrm{out},s_0}\setminus
 \operatorname{int}(\mathbb B_{\mathrm{in}}),\qquad
 F_W=W_{s_0}\cap Z.
\]
Identify (preserving the orientation) $\mathbb B_{\mathrm{out},s_0}$ with the closed unit $4$-ball
$D^4$. Since $\mathbb B_{\mathrm{in}}$ is a smooth $4$-ball lying in
the interior of $D^4$, we can always find an ambient diffeotopy fixed near $\partial D^4$ that carries
$\mathbb B_{\mathrm{in}}$ to the concentric ball $\tfrac12D^4$
(see for example \cite[Chapter~8, Theorem~3.1]{Hirsch} for a proof of this standard fact). So there are
orientation preserving diffeomorphisms
\[
 Z\cong D^4\setminus\operatorname{int}(\tfrac12D^4)
 \cong\mathbb S^3\times[0,1].
\]
Since $\mathbb B_{\mathrm{in}}$ contains all the singular points of
$W_{s_0}$, the map $\boldsymbol n_{s_0}$ identifies $F_W$ diffeomorphically with
$V_{s_0}\setminus\operatorname{int}(\boldsymbol n_{s_0}^{-1}(D_{\mathrm{in}}))$, preserving
orientations. Thus $F_W$ is connected, with oriented boundary
\[
 \partial F_W=-K_{\mathrm{tr}}\sqcup\Lambda_{s_0}.
\]
When we remove the interior of the disk
$\boldsymbol n_{s_0}^{-1}(D_{\mathrm{in}})$ from $V_{s_0}$
 the Euler characteristic is decreased by one, so
\[
 \chi(F_W)=\chi(V_{s_0})-1=(1-\mu(V,0))-1=-\mu(V,0). \qedhere
\]  
\end{proof}

\begin{remark}
\label{rem:tree-connected-sum}
We emphasize that the irreducibility of the germs $(W_{s_0},Q_j)$
is crucial for the {\em embedded realization} of the connected sum
\eqref{eq:transverse-knot}. It ensures that each $K_{Q_j}$ is a
knot and each $D_j$ is a disk. In \Cref{prop:exterior}, each handle
attachment along the tree removes a $3$-ball containing a trivial
arc from each of two boundary knot pairs and joins the knot
endpoints by the sides of the oriented band. The tree of handles and balls realizes the oriented connected sum on
$\partial\mathbb B_{\mathrm{in}}$ and joins the disks $D_j$ into
the single disk $\boldsymbol n_{s_0}^{-1}(D_{\mathrm{in}})$.

If one of the germs was reducible, then
$D_j=\boldsymbol n_{s_0}^{-1}(W_{s_0}\cap\mathbb B_j)$ would
consist of several disks. The $N-1$ bands would leave
$\boldsymbol n_{s_0}^{-1}(D_{\mathrm{in}})$ disconnected and the incoming boundary would
have several components. Adding further ambient $1$-handles to connect
all these disks would create cycles in the graph of balls and
handles, giving the ambient neighborhood nonzero first homology and it would no longer be a $4$-ball.
\end{remark}

\begin{corollary}
\label{cor:genus-identities}
There is a connected smooth oriented knot cobordism $(Z',F)$
from $(\mathbb{S}^3,K_{\mathrm{tr}})$ to
$(\mathbb{S}_W,T(m_X,m_X+1))$, with $Z'$ orientation preserving
diffeomorphic to $\mathbb{S}^3\times[0,1]$, such that
\begin{equation}
 g(F)=\delta(V,0)+\delta(W,0)-\frac{m_X(m_X-1)}2.
 \label{eq:composite-genus}
\end{equation}
Moreover,
\begin{equation}
 g_3(K_{\mathrm{tr}})
 =\sum_{j=1}^N\delta(W_{s_0},Q_j)
 =\delta(W,0)-\delta(V,0).
 \label{eq:seifert-genus-identities}
\end{equation}
\end{corollary}

\begin{proof}
Consider the identification (preserving orientation) of
\eqref{eq:boundary-identification} at $s=s_0$ used to identify the outgoing boundary
pair $(\mathbb S_{s_0},\Lambda_{s_0})$ of $(Z,F_W)$ from
\Cref{prop:exterior} with
$(\mathbb S_0,\Lambda_0)=(\mathbb S_W,\Lambda_W)$.
Glue $F_W$ to the surface $C\subset\mathbb S_W\times[0,1]$ from
\Cref{prop:braid} along
$\Lambda_{s_0}\simeq\Lambda_W\times\{0\}$, where we observe that the induced
boundary orientations are opposite.
The ambient manifold is
\[
 Z'=Z\cup_{\mathbb{S}_{s_0}\simeq \mathbb{S}_W}(\mathbb{S}_W\times[0,1]).
\]
By \Cref{prop:exterior}, the cobordism $Z$ is diffeomorphic to
$\mathbb{S}^3\times[0,1]$. And attaching to it the product collar $\mathbb{S}_W\times[0,1]$ does not change its
diffeomorphism type. The surface $F=F_W\cup C$ is connected,
smooth and oriented, and its boundary is
$-K_{\mathrm{tr}}\sqcup T(m_X,m_X+1)$.
Since $C$ has $b+1$ boundary components,
\Cref{prop:braid,prop:exterior} we get
\[
 \begin{aligned}
 \chi(F)
 &=\chi(F_W)+\chi(C)\\
 &=-\mu(V,0)+1-b-2\left(\delta(W,0)-\frac{m_X(m_X-1)}2\right)\\
 &=-2\left(\delta(V,0)+\delta(W,0)-\frac{m_X(m_X-1)}2\right).
 \end{aligned}
\]
Where the last equality is using the Milnor formula
$\mu(V,0)=2\delta(V,0)-b+1$. The surface $F$ has two boundary
components, so its Euler characteristic is $\chi(F)=2-2g(F)-2=-2g(F)$, which, combined with the equality above, yields
\eqref{eq:composite-genus}.

The Milnor fiber of the irreducible plane curve
$(W_{s_0},Q_j)$ is a fiber surface for $K_{Q_j}$. Therefore it has one
boundary component and so its Euler characteristic is
$1- \mu(W_{s_0},Q_j)=1-2\delta(W_{s_0},Q_j)$. Since fiber surfaces realize Seifert
genus, we get $
 g_3(K_{Q_j})=\delta(W_{s_0},Q_j).
$
The additivity of the Seifert genus under the connected sum
\eqref{eq:transverse-knot} and the delta identity
\eqref{eq:delta-specialization} at $s=s_0$ prove
\eqref{eq:seifert-genus-identities}.
\end{proof}

\section{Knot Floer inequalities and proof of the Reduction Theorem}
\label{sec:numerical-estimates}

Knot Floer homology was introduced independently by Ozsv\'ath and
Szab\'o \cite{OzsvathSzaboKnots} and Rasmussen
\cite{RasmussenThesis}. Using this theory, Ozsv\'ath, Stipsicz, and
Szab\'o introduced the invariant $\Upsilon$, with the aim of obtaining
a family of additive knot invariants that give lower
bounds for the slice genus \cite[p.~367]{OSS}.
For every knot, the slice genus is bounded above by the Seifert genus since
pushing the interior of a Seifert surface into the four-ball gives
a slice surface of the same genus.
Here we combine the bound for the slice genus given by the $\Upsilon$ with the equality between the
Seifert genus of an algebraic knot and the delta invariant of the
corresponding plane branch.
For two knots joined by a smooth oriented cobordism in
$\mathbb S^3\times[0,1]$, the genus of the cobordism gives a bound
for the difference between their $\Upsilon$ invariants. This is the crucial property of this invariant that we use in this work.

In \Cref{cor:genus-identities}, we have constructed a cobordism from the
connected sum $K_{\mathrm{tr}}$ of the knots corresponding to the transverse type singularities, to the torus
knot $T(m_X,m_X+1)$, and we have computed its genus. We apply the
$\Upsilon$ bound to this surface to obtain the comparison in
\Cref{cor:upsilon-comparison}. We then use the known properties of
$\Upsilon$ for algebraic knots and its explicit formula for these
torus knots to prove the lower bounds on $m_{\mathrm{tr}}$ in
\Cref{thm:estimates}. We compare the second lower bound with the upper
bound from \Cref{prop:polar-estimate} and use the resolution inequalities
\eqref{eq:pointbounds} to prove \Cref{thm:reduction}.

\subsection{The \texorpdfstring{$\Upsilon$}{Upsilon} invariant and cobordism bounds}
\label{subsec:shifted-upsilon}

For a knot $K\subset \mathbb{S}^3$, let
\[
 \Upsilon_K\colon[0,2]\cap\mathbb Q\longrightarrow\mathbb R
\]
denote the restriction to rational parameters of the invariant of
Ozsv\'ath, Stipsicz, and Szab\'o \cite[Section~3]{OSS}.   Inspired by Tru\"ol \cite[Remark~4.14]{Truol}, we associate to
every knot $K\subset \mathbb{S}^3$ the function
\begin{equation}
 \Phi_K\colon[0,2]\cap\mathbb Q\longrightarrow\mathbb R,
 \qquad
 \Phi_K(\tau)=\Upsilon_K(\tau)+g_3(K)\tau.
 \label{eq:phi-definition}
\end{equation}
The invariant $\Upsilon$ is additive under
connected sum \cite[Proposition~1.8]{OSS}, and so is the Seifert genus. Therefore,
\begin{equation}
 \Phi_{K_1\mathbin{\#}K_2}(\tau)
 =\Phi_{K_1}(\tau)+\Phi_{K_2}(\tau).
 \label{eq:phi-additivity}
\end{equation}

Let $g_4(K)$ denote the {\em slice genus} of $K$, that is,
the minimum genus of a connected smooth oriented surface properly
embedded in $D^4$ with boundary $K$.
See te proof of
\cite[Proposition~3.7]{FellerKrcatovich} for a similar argument to the one we use here.

\begin{lemma}
\label{lem:phi-cobordism}
Let $K_1,K_2\subset \mathbb{S}^3$ be oriented knots and let $(Z,F)$ be an
oriented knot cobordism from $(\mathbb{S}^3,K_1)$ to $(\mathbb{S}^3,K_2)$. Suppose
that $Z$ is orientation preserving diffeomorphic to
$\mathbb{S}^3\times[0,1]$.
Then, for $\tau\in[0,1]\cap\mathbb Q$,
\[
 \bigl(g_3(K_2)-g_3(K_1)-g(F)\bigr)\tau
 \leq \Phi_{K_2}(\tau)-\Phi_{K_1}(\tau)
 \leq \bigl(g_3(K_2)-g_3(K_1)+g(F)\bigr)\tau.
\]
\end{lemma}

\begin{proof}
Write $\overline K$ for the mirror of $K$ with its induced orientation,
and denote by $-K$ the knot $K$ with the opposite orientation.
Fix an orientation preserving diffeomorphism identifying $Z$ with $\mathbb S^3\times[0,1]$. Now, cut $F$ along a properly embedded arc joining its two
boundary components and remove an open tubular neighborhood
of that arc. In this way, we get a surface of genus $g(F)$ with boundary
$K_2\mathbin{\#}(-\overline{K_1})$.
By \Cref{rem:boundary-orientations}, the incoming boundary pair
with its induced boundary orientations is $(-\mathbb S^3,-K_1)$. A reflection of $\mathbb S^3$ reverses its
orientation and sends $K_1$ to $\overline{K_1}$, giving an orientation
preserving diffeomorphism of pairs
\[
 (-\mathbb S^3,-K_1)\cong(\mathbb S^3,-\overline{K_1}).
\]
An arc that joins the two boundary components of $\mathbb S^3\times[0,1]$
is ambient isotopic to $\{p\}\times[0,1]$ for some $p\in\mathbb S^3$.
So removing an open tubular neighborhood yields $D^3\times[0,1]$,
which is diffeomorphic to $D^4$ after rounding corners. Hence
\[
 g_4(K_2\mathbin{\#}(-\overline{K_1}))\leq g(F).
\]
For $\tau\in[0,1]\cap\mathbb Q$, the bound in
\cite[Theorem~1.11]{OSS}, additivity and the identity for mirror knots
\cite[Propositions~1.8 and~1.9]{OSS} together with invariance under reversal
of the orientation of the knot (see \cite[proof of Theorem~3.5]{OSS}) give
\[
 \left|\Upsilon_{K_2}(\tau)-\Upsilon_{K_1}(\tau)\right|
 =\left|\Upsilon_{K_2\mathbin{\#}(-\overline{K_1})}(\tau)\right|
 \leq g_4(K_2\mathbin{\#}(-\overline{K_1}))\tau
 \leq g(F)\tau.
\]
Equivalently,
\[
 -g(F)\tau
 \leq\Upsilon_{K_2}(\tau)-\Upsilon_{K_1}(\tau)
 \leq g(F)\tau.
\]
By \eqref{eq:phi-definition},
\[
 \Phi_{K_2}(\tau)-\Phi_{K_1}(\tau)
 =\Upsilon_{K_2}(\tau)-\Upsilon_{K_1}(\tau)
  +\bigl(g_3(K_2)-g_3(K_1)\bigr)\tau.
\]
Adding $\bigl(g_3(K_2)-g_3(K_1)\bigr)\tau$ to all three terms of
the preceding inequalities therefore gives
\[
 \bigl(g_3(K_2)-g_3(K_1)-g(F)\bigr)\tau
 \leq \Phi_{K_2}(\tau)-\Phi_{K_1}(\tau)
 \leq \bigl(g_3(K_2)-g_3(K_1)+g(F)\bigr)\tau,
\]
which proves the statement.
\end{proof}

\begin{corollary}
\label{cor:upsilon-comparison}
For $\tau\in[0,1]\cap\mathbb Q$,
\[
 \Phi_{T(m_X,m_X+1)}(\tau)
 \leq\sum_{j=1}^N\Phi_{K_{Q_j}}(\tau)+(m_X-M_{\boldsymbol{n}})\tau.
\]
\end{corollary}

\begin{proof}
We apply the upper bound of the inequality proven in \Cref{lem:phi-cobordism} to the
knot cobordism $(Z',F)$ of \Cref{cor:genus-identities} (where we are using
the identification $\mathbb{S}_W\cong \mathbb{S}^3$). The ambient manifold
of that cobordims is orientation preserving diffeomorphic to $\mathbb{S}^3\times[0,1]$. Write
$T=T(m_X,m_X+1)$ and
$g_T=g_3(T)=m_X(m_X-1)/2$. By
\eqref{eq:composite-genus} and \eqref{eq:seifert-genus-identities},
\[
 \begin{aligned}
 \Phi_T(\tau)-\Phi_{K_{\mathrm{tr}}}(\tau)
 &\leq\bigl(g_T-g_3(K_{\mathrm{tr}})+g(F)\bigr)\tau\\
 &=\bigl[g_T-(\delta(W,0)-\delta(V,0))
       +(\delta(V,0)+\delta(W,0)-g_T)\bigr]\tau\\
 &=2\delta(V,0)\tau.
 \end{aligned}
\]
By the additivity of $\Phi$ under connected sum \eqref{eq:phi-additivity} and by
\eqref{eq:transverse-knot}, we get
$
 \Phi_{K_{\mathrm{tr}}}(\tau)
 =\sum_{j=1}^N\Phi_{K_{Q_j}}(\tau).
$
Finally, $2\delta(V,0)=m_X-M_{\boldsymbol n}$ by definition
\eqref{eq:section-invariants}, proving the statement.
\end{proof}

\subsection{Numerical estimates}

\begin{theorem}\label{thm:estimates}
One has
\begin{equation}
 M_{\boldsymbol{n}}\leq m_{\mathrm{tr}},\qquad
 m_X+M_{\boldsymbol{n}}\leq3m_{\mathrm{tr}}.
 \label{eq:knot-estimates}
\end{equation}
Equivalently,
\begin{equation}
 \mu(V)\geq m_X-m_{\mathrm{tr}}-b+1,\qquad
 \mu(V)\geq2m_X-3m_{\mathrm{tr}}-b+1.
 \label{eq:milnor-estimates}
\end{equation}
\end{theorem}

\begin{proof}
Recall that $s_0\in\Delta_\eta\setminus\{0\}$ is a regular value of $v$,
chosen small enough and fixed in \Cref{subsec:exterior-cobordism}.
The points $Q_j=Q_j(s_0)$ for $1\leq j\leq N$ are the singular points
of $W_{s_0}$, and $K_{Q_j}$ is the link of the plane branch
$(W_{s_0},Q_j)$.
Equation~\eqref{eq:transverse-multiplicities} gives
$m_{\mathrm{tr}}\geq2$, and so we set
\begin{equation}
 \tau_0=\frac{2}{m_{\mathrm{tr}}}\in[0,1]\cap\mathbb Q.
 \label{eq:tau0}
\end{equation}
For $1\leq j\leq N$, put
$
 m_{Q_j}=\mult_{Q_j}(W_{s_0}).
$
If $Q_j\in\Sigma_i$, then \Cref{def:transverse-type} gives
$
 m_{Q_j}=m_i\leq m_{\mathrm{tr}}.
$
Since $(W_{s_0},Q_j)$ is a singular plane branch by
\Cref{prop:plane-sections}, we have that
$m_{Q_j}\geq2$ and its link $K_{Q_j}$ is an algebraic knot.
Wang's theorem \cite[Theorem~8]{Wang}, applied with
$a=m_{Q_j}$, gives
\[
 \Upsilon_{K_{Q_j}}(\tau)=-g_3(K_{Q_j})\tau
 \quad\text{for}\quad
 \tau\in\left[0,\frac{2}{m_{Q_j}}\right]\cap\mathbb Q.
\]
Equation~\eqref{eq:phi-definition} therefore gives
$\Phi_{K_{Q_j}}(\tau)=0$ on this interval.  Since
$\tau_0=2/m_{\mathrm{tr}}\leq2/m_{Q_j}$,
\begin{equation}
 \Phi_{K_{Q_j}}(\tau_0)=0
 \qquad\text{for }1\leq j\leq N.
 \label{eq:wang-vanishing}
\end{equation}
The inequality of \Cref{cor:upsilon-comparison}, together with \eqref{eq:tau0} and
\eqref{eq:wang-vanishing}, give
\begin{equation}
 \Phi_{T(m_X,m_X+1)}(\tau_0)
 \leq\frac{2(m_X-M_{\boldsymbol{n}})}{m_{\mathrm{tr}}}.
\label{eq:comparison}
\end{equation}
Now we need to bound the left-hand side of \eqref{eq:comparison}. The
torus knot $T(m_X,m_X+1)$ has Seifert genus $g_3(T)=m_X(m_X-1)/2$. Combining
\eqref{eq:phi-definition} with the computation of $\Upsilon$ in
\cite[Proposition~6.3]{OSS} gives, for
$i\in\{0,\ldots,m_X-1\}$ and
$\tau\in[2i/m_X,(2i+2)/m_X]\cap\mathbb Q$,
\begin{equation}
 \Phi_{T(m_X,m_X+1)}(\tau)=m_X\,i\tau-i(i+1).
 \label{eq:torus-phi}
\end{equation}
For $\tau\in[0,1]\cap\mathbb Q$, equation~\eqref{eq:torus-phi} implies
\begin{align}
 \Phi_{T(m_X,m_X+1)}(\tau)&\geq m_X\tau-2,
 \label{eq:torusone}\\
 \Phi_{T(m_X,m_X+1)}(\tau)&\geq2m_X\tau-6.
 \label{eq:torustwo}
\end{align}
Indeed, put $L_i(\tau)=m_X\,i\tau-i(i+1)$ for
$i=0,\ldots,m_X-1$. For
$\tau\in[2i/m_X,(2i+2)/m_X]\cap\mathbb Q$, the identity
\[
 L_{j+1}(\tau)-L_j(\tau)=m_X\tau-2(j+1)
 \qquad\text{for }j=0,\ldots,m_X-2
\]
shows that the sequence $L_0(\tau),\ldots,L_{m_X-1}(\tau)$ is
nondecreasing up to index $i$ and nonincreasing afterwards.
Hence \eqref{eq:torus-phi} gives
\[
 \Phi_{T(m_X,m_X+1)}(\tau)
 =\max_{0\leq i\leq m_X-1}\bigl\{m_X\,i\tau-i(i+1)\bigr\}
 \qquad\text{for }\tau\in[0,1]\cap\mathbb Q.
\]
Since \eqref{eq:resolution-multiplicity} gives
$m_X\geq m_{\boldsymbol n}^2\geq4$, retaining the terms
$i=1$ and $i=2$ gives \eqref{eq:torusone} and
\eqref{eq:torustwo}, respectively. Evaluating \eqref{eq:torusone} at $\tau_0$ and combining it with
\eqref{eq:comparison} gives
\[
 \frac{2m_X}{m_{\mathrm{tr}}}-2
 \leq\frac{2(m_X-M_{\boldsymbol{n}})}{m_{\mathrm{tr}}},
\]
which is equivalent to $M_{\boldsymbol{n}}\leq m_{\mathrm{tr}}$,
the first inequality in \eqref{eq:knot-estimates}.
Evaluating \eqref{eq:torustwo} at $\tau_0$ and combining it with
\eqref{eq:comparison} gives
\[
 \frac{4m_X}{m_{\mathrm{tr}}}-6
 \leq\frac{2(m_X-M_{\boldsymbol{n}})}{m_{\mathrm{tr}}},
\]
which is equivalent to
$m_X+M_{\boldsymbol{n}}\leq3m_{\mathrm{tr}}$,
the second inequality in \eqref{eq:knot-estimates}.
Milnor's formula and \eqref{eq:section-invariants} give
\[
 \mu(V)=2\delta(V,0)-b+1=m_X-M_{\boldsymbol{n}}-b+1,
\]
which translates the inequalities of \eqref{eq:knot-estimates} to the inequalities 
\eqref{eq:milnor-estimates}.
\end{proof}

\subsection{Proof of the Reduction Theorem}

\begin{proof}[Proof of \Cref{thm:reduction}]
Equations~\eqref{eq:differential} and \eqref{eq:order-n} give $m_{\boldsymbol{n}}\geq2$.
The second  estimate in \eqref{eq:knot-estimates} coming from knot theory and the
 estimate from the theory of plane curves in \Cref{prop:polar-estimate} bound the maximal
transverse multiplicity from below and above. More concretely, those yield
\[
 \frac{m_X+M_{\boldsymbol n}}{3}
 \leq m_{\mathrm{tr}}
 \leq 1+\frac{2m_X-M_{\boldsymbol n}-b}{m_{\boldsymbol n}}.
\]
Comparing the left hand side and right hand side, and multiplying by $3m_{\boldsymbol n}>0$ gives
\begin{equation}
 (m_{\boldsymbol{n}}-6)m_X+(m_{\boldsymbol{n}}+3)M_{\boldsymbol{n}}
 +3b-3m_{\boldsymbol{n}}\leq0.
 \label{eq:arithmetic}
\end{equation}
If $m_{\boldsymbol{n}}\geq6$, then the first term on the left side of
\eqref{eq:arithmetic} is nonnegative. The first  inequality in
\eqref{eq:pointbounds}, together with the fact that $b\geq1$, gives
\[
 (m_{\boldsymbol{n}}+3)M_{\boldsymbol{n}}+3b-3m_{\boldsymbol{n}}
 \geq m_{\boldsymbol{n}}^2+3>0.
\]
So the left side of \eqref{eq:arithmetic} is positive which is a contradiction.
If $m_{\boldsymbol{n}}=5$, we also use the second n inequality in
\eqref{eq:pointbounds} and get
\[
 -m_X+8M_{\boldsymbol{n}}+3b-15
 \geq3\bigl(M_{\boldsymbol{n}}+b-5\bigr)>0,
\]
because $5=m_{\boldsymbol{n}}\leq M_{\boldsymbol{n}}$ and $b\geq1$.
This again contradicts \eqref{eq:arithmetic}.
So $2\leq m_{\boldsymbol{n}}\leq4$.
\end{proof}

\printbibliography

\end{document}